\RequirePackage{fix-cm}
\documentclass[11pt]{amsart}     
\usepackage{geometry}
\usepackage{graphicx}
\usepackage[english]{babel}
\usepackage{amssymb}
\usepackage{amsaddr}
\usepackage{mathtools}
\usepackage{nicefrac}
\usepackage{graphicx}
\usepackage{subfigure}
\usepackage{booktabs}
\usepackage{tabularx}
\usepackage{verbatim}
\usepackage{multirow}
\RequirePackage[colorlinks,citecolor=blue,urlcolor=blue]{hyperref}
\usepackage{algorithm} \usepackage{algorithmic}
\usepackage{lipsum}
\usepackage{mathptmx}
\usepackage{latexsym}
\usepackage{tikz-cd}
\usepackage{nicefrac}
\usepackage{textgreek}

\makeatletter
\def\cl@chapter{\@elt {theorem}}
\makeatother
\usepackage{cleveref}
\crefformat{equation}{(#2#1#3)} 
\crefmultiformat{equation}{(#2#1#3)}%
{ and~(#2#1#3)}{, (#2#1#3)}{ and~(#2#1#3)}
\newtheorem{theoremmm}{Theorem}
\newtheorem{definition}{Definition}
\newtheorem{corollary}{Corollary}
\newtheorem{example}{Example}
\newtheorem{remark}{Remark}
\newtheorem{proposition}{Proposition}
\newtheorem{lemma}{Lemma}
\newtheorem{assumption}{Assumption}

\begin{document}

\title{On the connection between supervised learning and linear inverse  problems 
}


\author{Sabrina Guastavino \and Federico Benvenuto}

\address{Dipartimento di Matematica, Universit\`a degli Studi di Genova}

\email{guastavino@dima.unige.it \and benvenuto@dima.unige.it}

\thanks{This work was partially funded by the INdAM Research group GNCS.
The authors of this work have been funded by the European Union's Horizon2020 research and innovation programme under grant agreement No.640216 for the "Flare Likelihood And Region Eruption foreCASTing" (Flarecast) project.
}



\maketitle

\begin{abstract}
In this paper we investigate the connection between supervised learning and linear inverse problems. 
We first show that a linear inverse problem can be view
as a function approximation problem in a reproducing kernel Hilbert space (RKHS) and
then we prove that to each of these approximation problems corresponds a class of inverse problems. 
Analogously, we show that Tikhonov solutions of this class correspond to the Tikhonov solution of the approximation problem.
Thanks to this correspondence, we show that supervised learning and linear discrete inverse problems 
can be thought of as two instances of the approximation problem in a RKHS. 
These instances are formalized by means of a sampling operator which 
takes into account both deterministic and random samples and leads to discretized problems.
We then analyze the discretized problems and we study the convergence of their solutions 
to the ones of the approximation problem in a RKHS, both in the deterministic and statistical framework.
Finally, we prove there exists a relation between the convergence rates computed with respect to the noise level and the ones computed with respect to the number of samples. This allows us to compare 
upper and lower bounds given in the statistical learning and in the deterministic infinite dimensional inverse problems theory.


\end{abstract}




\section{Introduction}\label{section Introduction}

We consider the inverse problem of recovering a function $f$ such that 
\begin{equation}\label{Af=y}
y = Af
\end{equation}
where $f$ and $y$ are elements of Hilbert spaces and $A$ is a linear operator.
For estimating $f$ one can consider to have noisy infinite dimensional data, e.g. $y^\delta$ such that $\| y^\delta - y \| \leq \delta$,
or, more realistically, finite dimensional noisy samples $\{ y_1,\ldots y_n \}$ taken at points $\{ x_1,\ldots x_n \}$.
Solutions of inverse problems are usually achieved by using regularization methods, i.e. methods with specific convergence properties of the $L^2$-
norm of the error when $\delta$ goes to $0$. 

On the other hand, we consider the supervised learning problem of finding a function $g$ from a set of examples $\{ (X_i,Y_i) \}_{i=1,\ldots,n}$ randomly drawn from an unknown probability distribution $\rho$. The function $g$ has to explain the relationship between input-output, i.e. 
$$
Y_i \sim g(X_i)
$$
for all $i=1,\dots,n$ and $g(x)$ has to be a good estimate of the output when a new input $x$ is given.
In statistical learning the emphasis is on the consistency of estimators of $g$ and the convergence is required
in expectation (or in probability) when the number of examples $n$ goes to infinity. 

One of the main evidences of the connection occurring across these two problems is that regularization methods, usually developed for inverse problems, such as the spectral regularization, can be used for solving learning problems \cite{gerfo_spectral_2008,yao2007early}.
More in general, it is possible to apply regularization, e.g. $\ell_1$-penalized methods, for solving both learning and inverse problems \cite{guastavino2018consistent}. 
Furthermore, a recent trend is to use neural networks, a common tool for learning 
problems, for solving 
inverse problems \cite{mccann_convolutional_2017,adler2017solving}.
The fact that these problems can be solved using the same methods 
begs the question of to what extent they are similar and which are the key
points of the connection. 
In this paper we discuss the connection  at three levels: at the infinite dimensional level where both problems can be described as function approximation problems; at the discrete level where the two problems differentiate according to the nature of the data; at the level of convergence rates which have been considered under the same source conditions but different noise hypotheses.

\subsection{Our contribution}

One of the main differences between supervised learning and discrete inverse problems comes from the hypothesis on available data: 
in the first case, data are usually assumed to be given as the result of a {\it stochastic process} whose underlying distribution is {\it unknown}. 
On the contrary, in the second case, data are assumed to be given according to a {\it deterministic scheme}, at least for the independent variables and even when the dependent variables are assumed to be drawn in a stochastic manner, the underlying distribution is supposed to be {\it known}.
As the difference concerns hypotheses on the discrete data and our aim is to identify the key points of the connection, the starting point of this work is to consider the two problems at the infinite dimensional level.
Indeed, the fact that the range of a bounded operator is provided with a Reproducing Kernel Hilbert Space (RKHS) structure in a natural way (see \cite{kress_linear_2014}) allows us to describe the two problems as the same function approximation problem.
Here, we define the function approximation in a RKHS as an optimization problem. In particular, by introducing a suitable (non-linear) generalization of the 
Moore-Penrose inverse, we prove that the solution 
of an approximation problem in a RKHS 
can always be associated with a solution of a certain inverse problem.
Conversely, we prove that the set of solutions of a class of inverse problems 
corresponds to the solution of a certain approximation problem in a RKHS.
This set is defined up to the action of the unitary group.
Moreover, as in both frameworks Tikhonov regularization is used for obtaining 
stable solutions, we prove that there exists a correspondence between Tikhonov solutions of the approximation problem in a RKHS and of their corresponding inverse problems.

Then, we define a sampling operator for deriving both learning and inverse problems from the infinite dimensional formulation.
The peculiarity of this sampling operator is that it can take into account both deterministic and stochastic samples and it is an extension of the sampling operator defined in \cite{smale_shannon_2004}.
By means of this sampling operator, supervised learning and inverse problems can be thought of as two instances of the same infinite dimensional approximation problem.
We then analyze such discretized problems in both deterministic and stochastic frameworks and we study the convergence of their solutions to the ones of the original approximation problem.
The study of the convergence 
has to be conducted under different hypotheses according to the statistical or deterministic context.
In the statistical framework, the convergence is well-studied and 
is related to the 
argmax continuous theorem \cite{van1996weak}, which  guarantees the convergence of the solutions of the discretized problems to the ideal ones.
In the deterministic case, we show the convergence of solutions of the discrete inverse problems to ideal ones in a general setting by means of the fundamental Gamma convergence theorem 
\cite{braides2006handbook}.

Finally, we analyze the relationship between error convergence rates arising from the analysis of inverse problems where the error is introduced as a bounded infinite dimensional perturbation of the data, say $\delta$, and (inverse) learning problems where estimates of the solution are computed with a finite number of samples, say $n$.
We introduce an estimator starting from a finite number of random samples with two properties: first, under the usual source conditions, it shares the same optimal rates of the spectral regularization for learning problems, and second its rates are related to the ones of the classical spectral regularization for deterministic inverse problems. In particular,
for this estimator we prove an inequality between the infinite dimensional error given a certain degree of noise $\delta$ and the expected error given $n$ samples.
Such an inequality allows us to convert upper bounds with respect to the number of samples $n$ 
(typically analyzed in the statistical inverse framework \cite{vogel2002computational,blanchard2017optimal}) 
to upper bounds with respect to the noise level $\delta$ (typically analyzed in the deterministic inverse framework \cite{engl_regularization_1996}) and conversely lower bounds depending on $\delta$ to lower bounds depending on $n$.
A further contribution of this work is then to show that the optimal rate obtained in statistical learning is worse than the optimal one obtained in the deterministic error analysis, quantifying the difference between optimal rates in the two frameworks.

\subsection{Related works}

In the literature several authors proposed to solve learning problems by using regularization techniques originally developed for inverse problems, offering a glimpse of the connection between supervised learning and inverse problems
\cite{girosi_regularization_1995,smola_connection_1998,smale_shannon_2004,kurkova_neural_2005,cucker_learning_2007,mendelson_regularization_2010,wang_optimal_2011}.
In recent years, a rigorous formalization of this connection between supervised  learning and linear inverse problems have been proposed according to two strategies: the first considers the learning problem as an instance of an inverse one (see e.g. \cite{vito_learning_2005,gerfo_spectral_2008}) whereas the second introduces a bounded operator in the model equation of the statistical learning and it is known as inverse learning (see e.g. \cite{loustau2013inverse,blanchard2017optimal,rastogi_optimal_2017}).
The first strategy interprets a learning problem as an inverse one in which the forward operator is an inclusion and its main objective is to draw a connection between consistency in kernel learning and regularization in inverse problems, without dealing with convergence rates.
On the other hand, the second strategy considers inverse problems from a statistical estimation perspective highlighting the fact that statistical inverse problems can be thought of as learning problems starting from indirect data. In this case, under appropriate probabilistic source conditions, upper and lower bounds of convergence rates are provided for predictive and estimation error of spectral methods. 
The common thread between these two studies is to prioritize the learning context.
Indeed, in both these approaches data are samples randomly drawn from an unknown distribution, the typical assumption of the learning framework. 
However, as far as inverse problems are concerned, the theory is provided both in a statistical and infinite dimensional deterministic setting.
In \cite{bissantz2007convergence} a comprehensive study on the convergence rates with infinite dimensional deterministic and stochastic noise is provided.
Our analysis of the convergence rates of the proposed estimator is based on the results in this paper.

The paper is organized as follows.
In Section 2 we introduce the approximation problem in an infinite dimensional RKHS.
Then, we show that this infinite dimensional problem serves as an abstract prototype of linear inverse problems and supervised learning problems and we prove that such problems are equivalent up to the action of the unitary group. Finally, we prove that the Tikhonov regularized solutions in these two frameworks correspond to each other.
In Section 3 we focus on discretized problems by considering both the statistical and the deterministic framework and we analyze the convergence of the solution of discretized problems to their corresponding ideal ones, showing that the argmax continuous theorem and the fundamental theorem of Gamma convergence express the conditions for the convergence of the empirical solutions.
In Section 4 we give a result for converting error convergence rates with respect to the number of samples (in the statistical framework) to error convergence rates with respect to the noise level (in infinite dimensional inverse problems framework) and vice-versa.
In Section 5 we present the conclusions of our analysis.
\section{Infinite dimensional setting}\label{section Infinite dimensional setting}
Reproducing kernel Hilbert spaces arise in a number of areas, including statistical machine learning theory, approximation theory, generalized spline theory and inverse problems  \cite{cucker_learning_2007}. 
The usual definition of a Reproducing Kernel Hilbert Space (RKHS) is given for a Hilbert space of functions, as follows:
\begin{definition}
Let $\mathcal{H}$ be an Hilbert space of real valued functions on a non-empty set $\mathcal{X}$. $\mathcal{H}$ is said a reproducing kernel Hilbert space if for all $x\in\mathcal{X}$ the evaluation functional $L_x:f\in\mathcal{H}\to L_x(f):=f(x)$ is continuous.
\end{definition}
An important characterization of RKHSs, which can be even considered as an alternative definition, is the following:
$K:\mathcal{X}\times\mathcal{X}\to\mathbb{R}$ is a reproducing kernel of an Hilbert space $\mathcal{H}$ if for all $f\in\mathcal{H},$ $f(x)=<f,K_x>_{\mathcal{H}}$, where $K_x:=K(x,\cdot) \in \mathcal H$, $\forall$ $x\in\mathcal{X}$.
The definition of RKHS is not restricted to function spaces but allows us to consider reproducing kernels $K$ defined on $\mathcal{X}\times\mathcal{X}$, where $\mathcal{X}$ is a Borel set.
For function spaces $\mathcal{X}$ shall be $\mathbb{R}$ or $\mathbb{C}$, but in general it can be a countable set or a finite set \cite{aronszajn1950theory} (e.g. a pixel space) .
This perspective takes to see the reproducing kernel $K$ as function of two variables $(x,x')$, which can be continuous variables, e.g. $x,x'\in\mathbb{R}$, or can be represented by indexes $(i,j)$, e.g. countable variables $i,j\in\mathbb{N}$ or finite discrete variables $i,j\in\{1,\dots,n\}$.
In the latter case, the kernel $K$ is an infinite or finite matrix.
We now define an approximation problem for functions, sequences or vectors, by requiring that the solution belongs to a suitable RKHS.

\subsection{Approximation problems in RKHS}\label{Approximation problems in RKHS}

We introduce the approximation problem in a RKHS as the problem of finding the closest element of the RKHS to a given one. 
Let us call $y$ the element to approximate in a given Hilbert space $\mathcal{H}_2$ and let $\mathcal{H}_K\subseteq\mathcal{H}_2$ be a RKHS with reproducing kernel $K$.
We define the solution of the approximation problem as the minimizer of a functional $R_y: 
\mathcal{H}_2\to\mathbb{R}$ over the RKHS $\mathcal{H}_K$, i.e.
\begin{equation}
\label{approx_problem}
g_{R_y} := \arg\min_{g\in \mathcal{H}_K} R_y(g).
\end{equation}
The idea is that $R_y$ measures the approximation error. 
We require that $R_y(g) \ge 0$ for all $g\in \mathcal{H}_2$, and
$R_y(g) = 0$ iff $g=y$.
Under these hypotheses, if $y\in\mathcal{H}_K$ the existence and uniqueness are assured by requiring that $R_y$ is strictly convex. Otherwise, if $y\notin\mathcal{H}_K$ the existence and uniqueness are assured either by requiring that
\begin{itemize}
\item[a)] $R_y$ is lower semicontinuous, strictly convex and coercive with respect to the norm $\| \cdot\|_{ \mathcal{H}_2}$ and $\mathcal{H}_K\subseteq\mathcal{H}_2$ is closed, or
\item[b)] $R_y$ is lower semicontinuous, strictly convex and coercive with respect to the norm $\| \cdot\|_{ \mathcal{H}_K}$.
\end{itemize}
A typical example is $R_{y}(g)=\Vert y-g\Vert^2_{\mathcal{H}_2}$ with $\mathcal{H}_K$ closed in $\mathcal{H}_2$.

\subsection{Linear inverse problems in Hilbert spaces}\label{sub:Linear inverse problems in Hilbert spaces}

Let $\mathcal{H}_1$ 
be an Hilbert space (generally different from $\mathcal{H}_2$) and $A$ a bounded linear operator $A: \mathcal{H}_1\to\mathcal{H}_2$. The inverse problem associated to the operator $A$ consists in finding $f$ satisfying equation \cref{Af=y}
given $y\in\mathcal{H}_2$.
The ill-posedness of inverse problems leads to the definition of the generalized solution, usually denoted by $f^{\dagger}$, which, from a variational point of view, can be seen as the minimal norm solution of the least squares problem $\min_{f\in\mathcal{H}_1} \|y-Af\|^2_{\mathcal{H}_2}$.
The variational form of the generalized inverse suggests that a  strategy for approximating the solution of an inverse problem is to minimize a functional along $f$. Then, we consider the set of $R_y$-minimum solutions of the problem \cref{Af=y} defined by
\begin{equation}
\label{f_R}
\mathcal{S}_{A,R_y} := \arg \min_{f\in\mathcal{H}_1} R_y(Af)
\end{equation}
and take the minimum norm solution.
When at least an $R_y$-minimum solution $f_{R_y}$ exists, $\mathcal{S}_{A,R_y}$ is the affine subspace given by $f_{R_y} + Ker(A)$, where $Ker(A)$ denotes the nullspace of $A$.
\begin{definition}
$f^{\dagger}_{R_y} \in \mathcal H_1$ is called the $R_y$-generalized solution of the inverse problem \cref{Af=y} if it is the $R_y$-minimum solution of \cref{Af=y} with minimum norm, i.e.
\begin{equation}
\label{fdaggerR}
f_{R_y}^{\dagger} = \arg\min_{f\in\mathcal{S}_{A,R_y}} \| f\|_{\mathcal{H}_1}.
\end{equation}
\end{definition}
As in \cref{Approximation problems in RKHS} we require that $R_y(g) \ge 0$ for all $g\in \mathcal{H}_2$, and
$R_y(g) = 0$ iff $g=y$.
We discuss some hypotheses which assure the existence and uniqueness of the $R_y$-generalized solution. We denote with $\Im(A)$ the range of $A$. Under these hypotheses, if $y\in\Im(A)$ the existence and uniqueness are assured by requiring that $R_y$ is strictly convex. Otherwise, if $y\notin\Im(A)$ the existence and uniqueness are assured either by requiring that
\begin{itemize}
\item[a)] $R_y$ is lower semicontinuous, strictly convex and coercive with respect to the norm $\| \cdot\|_{ \mathcal{H}_2}$ and $\Im(A)\subseteq\mathcal{H}_2$ is closed, or
\item[b)] $f\in\mathcal{H}_1 \mapsto R_y(Af)$ is lower semicontinuous, strictly convex and coercive with respect to the norm $\| \cdot\|_{ \mathcal{H}_1}$.
\end{itemize}

When $R_y$ is different from the least squares functional, this procedure provides a generalization of the so-called Moore Penrose generalized solution.
Such a generalization is needed to develop the equivalence between approximation problems in RKHSs and classes of linear inverse problems. We introduce it in the next paragraph.

\subsection{Equivalence between problems}\label{section Equivalence between problems}

We show the equivalence between an approximation problem in a RKHS and an inverse problem by proving that there is a natural correspondence of the solutions of the two problems.
We make use of the following:
\begin{assumption}
\label{assumption1}
Let $\mathcal{H}_1$ be a real separable Hilbert space and $\mathcal{H}_2$ be a real Hilbert space on a Borel space $\mathcal{X}$.
For all $x\in\mathcal{X}$ and for all $f\in\mathcal{H}_1$ there exists a constant $c>0$ such that 
\begin{equation}\label{assumption on A}
|Af(x)|\le c\| f\|_{\mathcal{H}_1}.
\end{equation}
\end{assumption}

The \cref{assumption1} together with the Riesz's representation theorem implies that for all $x$ there exists an element $\phi_x\in\mathcal{H}_1$ such that
\begin{equation}
\label{scalar-product}
(Af)(x) = <f, \phi_x>_{\mathcal{H}_1}.
\end{equation}
Moreover, it is well known that the range of the operator $A$
is a RKHS (e.g. see \cite{steinwart_support_2008}).
The following proposition is an adaptation of this result to our context.

\begin{proposition}\label{RKHS}
$\Im(A)$ equipped with the norm
\begin{equation*}
\| g\|_{\mathcal{H}_K} = \min\{ \| w\|_{\mathcal{H}_1} : w\in\mathcal{H}_1 \text{ s.t } g(x)=<w,\phi_x>_{\mathcal{H}_1},  ~ x\in\mathcal{X}\}
\end{equation*}
is a RKHS 
with kernel
\begin{equation}
\begin{split}
\label{Kernel}
K: &  \mathcal{X}\times \mathcal{X}\to \mathbb{R} \\ 
&  (x,r) \to K(x,r):= <\phi_x,\phi_r>_{\mathcal{H}_1}.
\end{split}
\end{equation}
\end{proposition}
We remark that $K$ by definition is a positive semi-definite kernel over $\mathcal{X}$ and $\phi$ represents the feature map on the feature space $\mathcal{H}_1$. Furthermore we have
$$\Im(A) = \overline{span\{ K_x, ~ x\in\mathcal{X}\}}.$$ 
Moreover, it is worth observing that conditions usually required on a reproducing kernel and on its associated RKHS are satisfied:
$\mathcal{H}_K$ is separable since $\mathcal{H}_1$ is separable and $A$ is a partial isometry from $\mathcal{H}_1$ to $\Im(A)$, and for all $x\in\mathcal{X}$ $K(x,x)\le c^2$ since \cref{assumption1} applies.

Now we introduce the restriction of $A$ to the space orthogonal to its kernel and we prove the main result of this section
which identifies the solutions of the two problems $g_{R_y}$ and $f^{\dagger}_{R_y}$ as defined in equations \cref{approx_problem} and \cref{fdaggerR}, respectively.
We denote with $\tilde{A}$ the restriction operator, i.e.
\begin{equation*}
\tilde{A}:=A_{|Ker(A)^{\perp}}: Ker(A)^{\perp}\to \Im(A).
\end{equation*}
By definition, $\tilde{A}$ admits the inverse operator $\tilde{A}^{-1}$.

\begin{theoremmm}
\label{correspondence}
Let $g_{R_y}$ be the solution of the approximation problem in the RKHS $\mathcal{H}_K$ with kernel $K$ defined in equation \cref{approx_problem}.
Let $f_{R_y}^\dagger$ be the solution of the inverse problem defined in equation \cref{fdaggerR} with the operator $A$ defined in equation \cref{scalar-product}.
If $\forall$ $x,x'\in\mathcal{X}$ $K(x,x') = <\phi_x,\phi_{x'}>_{\mathcal{H}_1}$, we have
\begin{equation}\label{cor-fdagger-g}
f^{\dagger}_{R_y} =\tilde{A}^{-1} g_{R_y} ~.
\end{equation}
\end{theoremmm}
\begin{proof}
By hypothesis we have the following identification $\Im(A)=\mathcal{H}_K$ intended as RKHSs. Thanks to this identification the hypotheses on $R_y$ in  problems \cref{approx_problem,fdaggerR} are exactly the same: the hypotheses a) in \cref{Approximation problems in RKHS} and in \cref{sub:Linear inverse problems in Hilbert spaces} are trivially the same hypothesis, the hypotheses in b) in \cref{Approximation problems in RKHS} and in \cref{sub:Linear inverse problems in Hilbert spaces} are equivalent by noting that the coercivity of $R_y$ with respect to the norm $\|\cdot\|_{\mathcal{H}_K}$ corresponds to the coercivity of $f\mapsto R_{y}(Af)$ with respect to the norm $\|\cdot\|_{\mathcal{H}_1}$.
Let $g_{R_y}$ be the solution of the problem \cref{approx_problem} and let $\tilde{f} := \tilde{A}^{-1} g_{R_y}$.
Then for all $f \in \mathcal{H}_1$ we have
\begin{equation}
\label{2}
R_y(Af)  \ge \min_{g\in \Im(A)} R_y(g) = R_y(g_{R_y}) = R_y(A \tilde{f} ) ~ ,
\end{equation} 
i.e. $\tilde{f}$ is solution of problem \cref{f_R}. Furthermore, by definition of $\tilde{A}^{-1}$, $\tilde{f} \in Ker(A)^{\perp}$ and therefore $\tilde{f}$ is the solution of \cref{fdaggerR}, that is  $\tilde{f} = f^{\dagger}_{R_y}$.
\end{proof}
\begin{remark}
Under \cref{assumption1}, given an inverse problem described by a linear operator $A$ (characterized by a map $\phi$), it is always possible to associate with it an approximation problem in the RKHS $\mathcal{H}_K$ with kernel $K$ defined by the map $\phi$, i.e. $
K(x,x')=<\phi_x,\phi_{x'}>_{\mathcal{H}_1}$ for all $x,x'\in\mathcal{X}$.
\end{remark}
\begin{remark}
Given an approximation problem in the RKHS $\mathcal{H}_K$ with kernel $K$, it is always possible to associate with it a feature map $\phi:x\in\mathcal{X}\to\phi_x\in\mathcal{H}_1$, where $\mathcal{H}_1$ is an Hilbert space and such that $K(x,x')=<\phi_x,\phi_{x'}>_{\mathcal{H}_1}$ for all $x,x'\in\mathcal{X}$. In such a way we define $\mathcal{F}=\overline{span\{\phi_x ~, ~ x\in\mathcal{X}\}}$, which is the feature space, and an inverse problem whose operator $A$ is given in equation \cref{scalar-product}. By construction we have the identification between the feature space and the orthogonal of the kernel of the operator, i.e. $\mathcal{F}=Ker(A)^{\perp}$.
In the case that $K$ is a continuous reproducing kernel, the Mercer theorem gives us the way to describe the feature map $\phi$ and the feature space is  $\ell_2$, while in the general case (when $K$ is not necessarily continuous) we can consider the canonic feature map, that is $\phi:\mathcal{X}\to \mathcal{H}_K$ where $\forall$ $x\in\mathcal{X}$ $\phi_x = K_x$. 
\end{remark}

From the second remark the feature map associated with a given kernel $K$ is determined up to the action of unitary group on $\mathcal{H}_1$, i.e.
\begin{equation}
\label{K}
K(x,x')=<\phi_x,\phi_{x'}>_{\mathcal{H}_1}=<U\phi_x,U\phi_{x'}>_{\mathcal{H}_1},
\end{equation}
for each unitary operator $U$ acting on $\mathcal{H}_1$. 
In particular,
we can define an equivalence relation $\sim$ on $\mathcal{H}_1$ using the left action of the unitary group 
$\mathcal{U}$. Let $f,f' \in \mathcal{H}_1$
\begin{equation}
\label{quotient}
f ~ \sim ~f' \quad \iff \quad \exists ~ U \in\mathcal{U} ~~~|~~~ f' = U f .
\end{equation}
We can also define an equivalence $\sim^{\mathcal{X}}$ between feature maps. Let $\phi,\phi' \in \mathcal{H}_1^{\mathcal{X}}$
\begin{equation}
\label{feature_map_quotient}
\phi ~ \sim^{\mathcal{X}} ~\phi' \quad \iff \quad \phi_x \sim \phi'_x ~, ~ \forall ~ x \in \mathcal{X}.
\end{equation}
Then, we define the map 
\begin{eqnarray*}
\mathcal{K} : \mathcal{H}_1^\mathcal{X}  & \to & \mathbb{R}^{\mathcal{X} \times \mathcal{X}} \\
\phi & \longmapsto & K_\phi
\end{eqnarray*}
with $K_\phi(x,x')=\langle \phi_x,\phi_{x'} \rangle_{{\mathcal H}_1}$.
Therefore, from equations \cref{K} and \cref{feature_map_quotient} we have a bijection
\begin{eqnarray*}
\mathcal{H}_1^\mathcal{X} / \sim^{\mathcal{X}}
& \longleftrightarrow
& \Im(\mathcal{K}) 
\subset 
\mathbb{R}^{\mathcal{X} \times \mathcal{X}} \\ \nonumber
\bar \phi 
& \longleftrightarrow
& K_{\phi}
~,
\end{eqnarray*}
where $\bar \phi$ is the class induced by the equivalence relation $\sim^{\mathcal{X}}$ in \cref{feature_map_quotient}.
We denote with $A_{\phi}$ the operator defined in equation \cref{scalar-product}. We have 
\begin{equation*}
g_{R_y} = A_{\phi} f^{\dag}_{R_y} = A_{\phi'} (f_{R_y}^{\dag})',
\end{equation*}
where $\phi\sim^{\mathcal{X}}\phi'$ and $f^\dag_{R_y}\sim (f^\dag_{R_y})'$.
Then we also have a bijection
\begin{eqnarray*}
\mathcal{H}_1 / \sim ~ 
& \longleftrightarrow
& \mathcal{H}_K \\
\overline{f^\dagger_{R_y}}
& \longleftrightarrow
& g_{R_y}
\end{eqnarray*}
stating that, for any ${R_y}$ satisfying conditions of problem \cref{approx_problem} (or equivalently \cref{fdaggerR}) and for any $y \in \mathcal{H}_2$, the class of $R_y$-generalized solutions $\overline{f^\dagger_{R_y}}$ corresponds to the solution $g_{R_y}$ of the approximation problem in the RKHS $\mathcal{H}_K$ \cref{approx_problem}.
Let us now fix an element $y \in \mathcal{H}_2$ and a functional ${R_y}$.
For each $K \in \Im(\mathcal{K})$ we define the function
$T_{R_y}(K) := g_{R_y}$
which maps the kernel $K$ to the solution of the approximation problem in a RKHS \cref{approx_problem}.
In the same way, for each $\phi \in \mathcal{H}_1^\mathcal{X}$ we define the function $S^\dagger_{R_y}(\phi) := f^\dagger_{R_y}$
which maps the feature map $\phi$ to the $R_y$-generalized solution of the inverse problem \cref{Af=y}.
Then, for each class $\overline{\phi}$, we can define a map $
\overline{S^\dag_{R_y}}: \mathcal{H}_1^\mathcal{X} /  \sim^{\mathcal{X}} \to  \mathcal{H}_1 / \sim$ as follows
\begin{equation}
\label{class_map}
\overline{S^\dag_{R_y}}(\overline{\phi}) :=  
\pi ( S^\dag_{R_y} ( \phi ) ) ~ ,
\end{equation}
where $\phi$ is a representer of $\overline{\phi}$ and $\pi$ is the quotient map with respect to the equivalence relation $\sim$ in \cref{quotient}. 
This definition is well-posed since it does not depend on the choice of the representer $\phi$.
We can summarize this discussion with the commutative diagram in \cref{equivalence-diagram}.
\begin{figure}
\[
\begin{tikzcd}[ampersand replacement=\&,column sep=small]
  \phi \in \mathcal{H}_1^\mathcal{X} \arrow[pos=0.75]{rr}{\pi^{\mathcal{X}}}
     \arrow[swap]{dr}{\mathcal{K}}
     \arrow[swap]{dd}{S^\dagger_{R_y}}
     \&      \& 
  \overline{\phi} \in \mathcal{H}_1^\mathcal{X} /  \sim^{\mathcal{X}}
  \arrow[swap,swap]{ld}{1-1}
  \arrow{dd}{\overline{S^\dag_{R_y} }} \\
     \& K_\phi \in \Im(\mathcal{K}) 
    \arrow[pos=0.75,crossing over]{dd}{T_{R_y}}
     \& \\
f_{R_y}^\dagger \in \mathcal{H}_1
\arrow[pos=0.75]{rr}{\pi}
     \arrow[swap]{dr}{A_\phi}
     \&      \& \overline{f_{R_y}^\dagger} \in \mathcal{H}_1 / \sim \\
     \& g_{R_y} \in \mathcal{H}_K 
    \arrow[swap]{ru}{1-1} \&
\end{tikzcd}
\]
\caption{Commutative diagram summarizing the equivalence between approximation in a RKHS and linear inverse problems.}\label{equivalence-diagram}
\end{figure}
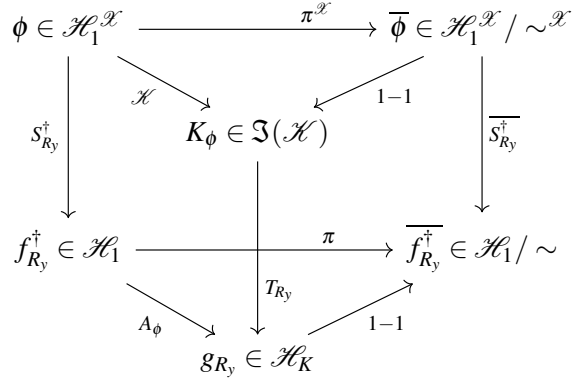
In synthesis, when an approximation problem in a RKHS  is provided with a feature map, it is equivalent to a linear inverse problem. 
If a feature map is not given, we can associate with the approximation problem in a RKHS as many inverse problems as feature maps (and so features spaces) which give rise to the same kernel.

\subsection{Equivalence between methods}\label{section Equivalence between methods}
When $y$ is corrupted by noise, the inverse problem needs to be addressed in a different way as the ${R_y}$-generalized solution $f^\dag_{R_y}$ may not exist or it may not depend continuously on the data.
A well-known strategy common to both approximation and inverse problems is Tikhonov regularization \cite{engl_regularization_1996}.
It allows us to find solutions of the problem which depend continuously on the data by re-stating the approximation problem in RKHS $\mathcal{H}_K$ defined in equation \cref{approx_problem} as follows
\begin{equation}
\label{approx_probl_reg}
\hat{g}_{{R_y},\lambda} = \arg\min_{g\in\mathcal{H}_K} R_y( g)+\lambda \Omega(g)  ,
\end{equation}
and the inverse problem associated to the operator $A$ given data $y$ defined in equation \cref{fdaggerR} as follows
\begin{equation}
\label{inv_probl_reg}
\hat{f}_{{R_y},\lambda} = \arg\min_{f\in\mathcal{H}_1} R_y( Af)+\lambda \Omega(f) ~ . 
\end{equation}
In these generalized Tikhonov regularization schemes $R_y$ is usually called the data fidelity term, $\Omega$ is the penalty term and $\lambda>0$ is the regularization parameter. 
The purpose of the penalty term is to induce stability and to allow the incorporation of a priori information about the desired solution according to the magnitude of the parameter $\lambda$.
In this context we assume that the penalty term has the following form
\begin{eqnarray}
\Omega(h) := \psi(\| h\|_{\mathcal{H}}) ~ ,
\end{eqnarray}
where $\psi:[0,+\infty)\to\mathbb{R}_+$ is a continuous convex and strictly monotonically increasing real-valued function, $h$ is an element of an Hilbert space $\mathcal{H}$ and $\| \cdot \|_{\mathcal{H}}$ denotes its norm. 
Now we show that the result of \cref{correspondence} 
can be extended to the case of Tikhonov regularized solutions $\hat{f}_{R_y,\lambda}$ and $\hat{g}_{R_y,\lambda}$. 
\begin{theoremmm}
\label{theorem methods}
Under the same assumptions of \cref{correspondence} we have
\begin{equation}
\label{corr-f-g}
\hat{f}_{R_y,\lambda}=\tilde{A}^{-1} \hat{g}_{R_y,\lambda} ~.
\end{equation}
\end{theoremmm}
\begin{proof}
As in the proof of the \cref{correspondence} we have the identification $\Im(A) = \mathcal{H}_K$ as RKHSs and the hypotheses on functionals to minimize in \cref{approx_probl_reg,inv_probl_reg} are the same.
Let $\tilde{f} := \tilde{A}^{-1} \hat{g}_{{R_y},\lambda}$.
By definition of $\tilde{A}^{-1}$, $\tilde{f} \in Ker(A)^{\perp}$ and so $\| \hat{g}_{{R_y},\lambda}\|_{\mathcal{H}_K}=\| \tilde{f} \|_{\mathcal{H}_1}$. 
For all $f \in \mathcal{H}_1$ we have
\begin{equation}
\label{22}
\begin{split}
R_y(Af) +\lambda \psi(\| f\|_{\mathcal{H}_1}) & \ge \min_{g\in \Im(A)} R_y(g) +\lambda \psi(\| g\|_{\mathcal{H}_k})\\
& = R_y(\hat{g}_{{R_y},\lambda}) +\lambda \psi(\| \hat{g}_{{R_y},\lambda} \|_{\mathcal{H}_k})\\
& = R_y( A \tilde{f} ) +\lambda \psi(\| \tilde{f} \|_{\mathcal{H}_1})
\end{split}
\end{equation}
i.e. $\tilde{f}$ is solution of problem \cref{inv_probl_reg}. This concludes the proof.
\end{proof}
As in the case of $R_y$-generalized solutions, we have a commutative diagram for Tikhonov regularized solutions.
The diagram has exactly the same shape of the one shown in \cref{equivalence-diagram} but arrows and nodes refer to the solution of problems \cref{approx_probl_reg,inv_probl_reg}.
In particular, we have to replace:
$T^{\dagger}_{R_y}$ with the function $T_{{R_y},\lambda}(K) := \hat{g}_{{R_y},\lambda}$ which maps the kernel $K$ to the Tikhonov solution \cref{approx_probl_reg};
$S^\dagger_{R_y}$ with the function $S_{R_y,\lambda}(\phi) := \hat{f}_{R_y,\lambda}$ which maps the feature map $\phi$ to the Tikhonov solution \cref{inv_probl_reg};
$\overline{S^{\dagger}_{R_y}}$ with the map $\overline{S_{{R_y},\lambda}}$  defined as in equation \cref{class_map} by substituting $S_{R_y}^\dagger$ with $S_{{R_y},\lambda}$;
$\overline{f^\dagger_{R_y}}$ with $\overline{\hat{f}_{{R_y},\lambda}}$, which is the class of Tikhonov solutions corresponding to the Tikhonov solution of the approximation problem in the RKHS represented by $K_\phi$.

\section{Discrete data}
\label{discrete-data}
The purpose of this section is to show that some applied problems, namely discrete inverse problems, interpolation problems and statistical (inverse) learning, despite appearing different, can be though of as instances of the approximation problem in a RKHS \cref{approx_problem}. 
To this end, we introduce a suitable discretization operator 
mapping the infinite dimensional data $y$ to a finite number of samples together with a specific form of the functional ${R_y}$.
The idea of the discretization operator is to 
consider, in place of the data $y$, a set of samples $\{(X_i,Y_i)\}_{i=1}^n$ statistically or deterministically related to $y$.
In this way we will retrieve the formulation of various applied problems by minimizing the empirical form of the ideal functional $R_y$.
To realize the discretization operator, i.e. a map from $\mathcal H_2$ to a sample space, we proceed as follows.
Let us consider the set $\mathcal{P}$ of all possible Borel probability distributions over a compact space $\mathcal Y\subseteq\mathbb{R}$ and let $F_V:\mathcal{P}\to\mathbb R$ be a function defined by
\begin{equation}\label{F_V}
F_V(\tilde{\rho}) := \arg\min_{w\in\mathbb{R}} \int_{\mathcal{Y}} V(Y,w) ~ d\tilde{\rho}(Y),
\end{equation}
where $V$ is called {\it loss function} in the statistical learning terminology \cite{rosasco2004loss}. The function $F_V$ is defined provided that $V: \mathcal{Y}\times\mathbb{R}\to [0,+\infty)$ is measurable and integrable with respect to the first variable and 
$V(Y,\cdot)$ is lower semicontinuous, strictly convex and coercive $\forall$ $Y\in\mathcal{Y}$. 
Given a function $V$, $F_V(\tilde{\rho})$ can represent a characteristic of the distribution $\tilde{\rho}$: by instance, if $V$ is the square loss usually used in regression problems, i.e. $V(Y,w)=(w-Y)^2$, or $V$ is the Kullback-Leibler divergence, then $F_{V}(\tilde{\rho})$ is the expected value,  i.e. $F_{V}(\tilde{\rho})=\mathbb{E}(Z)$, where $Z$ is a random variable with probability distribution $\tilde{\rho}$; if $V$ is the square loss usually used in classification problems, i.e. $V(Y,w)=(1-Yw)^2$ then $F_{V}(\tilde{\rho})=\mathbb{E}(Z)/\mathbb{E}(Z^2)$; if $V$ is the absolute value loss, i.e. $V(Y,w)=|w-Y|$ then $V$ is the median of the distribution $\tilde{\rho}$. 

We now want to define a map from $\mathbb{R}$ to $\mathcal{P}$, roughly speaking an inverse of $F_V$. We 
introduce an application 
\begin{equation}\label{varrho}
\begin{split}
\vartheta:\mathbb{R} & \to \mathcal{P}\\
 z & \to \tilde{\rho}_z,
\end{split}
\end{equation}
mapping $z\in\mathbb{R}$ in a distribution $\tilde{\rho}_z$ such that $F_V\circ\vartheta=id$. Given a function $y$, $\vartheta$ maps $y(x)$ to a distribution $\tilde{\rho}_{y(x)}$ such that $y(x)$ is the characteristic of $\tilde{\rho}_{y(x)}$ for each $x\in\mathcal{X}$. 
Therefore we define the following sampling operator 
\begin{equation}
\label{sampling operator}
\begin{split}
S^{(n)}_{\bar x, \vartheta} : \mathcal{H}_2 & \to \mathcal{Y}^n\\
 y & \to ( Y_i )_{i=1,\dots,n}
\end{split}
\end{equation}
where each $Y_i$ is drawn from the distribution 
$\tilde{\rho}_{y(x_i)}:=\vartheta(y(x_i))$
and the set of points 
$\bar{x}=\{x_1,\dots,x_n\}\subset\mathcal{X}$ 
can be either given a priori (in a deterministic manner) or drawn from a probability distribution 
$\nu$ over $\mathcal X$.
Once $V$ is fixed, for any chosen sampling $S^{(n)}_{\bar x, \vartheta}$ let us consider the functional defined as 
\begin{equation}
\label{general-problem}
R_{y}(g) 
:=
\int_{\mathcal X \times Y} V(Y,g(X)) d\tilde{\rho}_{y(X)}(Y) d\nu(X)
\end{equation}
which depends on $\nu$ and $\vartheta$ as well as on $y$ and on $V$. 
When the set of sampling points $\bar{x}$ is given in a deterministic manner, we assume $d\nu(x) = dx$.
From now on, we will denote the distribution $\tilde{\rho}_{y(x)}=\rho(\cdot|X=x)$ since it will represent the conditional distribution with respect to $X=x$. 
Henceforth, we consider the approximation problem in a RKHS \cref{approx_problem} with functional $R_y$ given in \cref{general-problem}. By applying $S^{(n)}_{\bar x, \vartheta}$ to the data $y$, we now show that we can retrieve the formulation of different applied problems according to 
whether $\rho$ and $\nu$ are known or not and, if they are known, according to their specific explicit form.
In general, when just a finite number of sample values is known, all these problems are addressed by minimizing the following empirical form of the functional, i.e.
\begin{equation}
\label{general-discrete-problem}
R_{\mathcal Z_n}(g) := 
\frac{1}{n} \sum_{i=1}^n V((S^{(n)}_{\bar{x},\vartheta}(y))_i,g(x_i)) ~.
\end{equation}
In particular, when $\rho(\cdot|\cdot)$ and $\nu$ are not known we retrieve the formulation of statistical learning problems, 
while if $\rho(\cdot|\cdot)$ and $\nu$ are given 
we have the following cases. 
\begin{itemize}
\item[i)] Stochastic case: the samples $Y_i$ are drawn from $\rho(Y|X=x_i)$ and $x_i$ are the elements of $\bar x$, given at random according to a distribution $\nu$. In this case 
we can describe inverse regression problems with random matrix design.
\item[ii)] Semi-stochastic case: the samples $Y_i$ are drawn from a generic probability distribution $\rho(Y|X=x_i)$ and $x_i$ are given not at random. In this case we assume $d\nu(x)=dx$. 
This is the setting used for describing discrete inverse problems with random noise under the maximum likelihood approach or inverse regression problems with deterministic matrix design. 
\item[iii)] Deterministic case: the samples $Y_i$ are the values of the function $y$ at the points $x_i$, given not at random. This can be thought of as the samples $Y_i$ are drawn from 
$\rho(Y|X=x_i) = \delta( Y - y(X) | X=x_i )$. In this case we assume $d\nu(x)=dx$ and 
the sampling operator can be denoted by $S^{(n)}_{\bar x}$ as in \cite{smale_shannon_2004}.
This is the setting used for describing interpolation problems or discrete inverse problems with deterministic noise when we consider a noisy version $y^{\delta}$ of $y$.
\end{itemize}
The crucial difference between the first and the latter two cases is that in the first case discretization has to be defined according to a stochastic process while in the second and third cases at least a part of the discretization is usually defined in a deterministic manner. 
Incidentally, we notice that in learning problems a given point can be sampled more than once whereas in inverse problems each sample $x \in \mathcal X$ is usually taken once.
For example, in a machine learning problem the samples can be view as the result of a sampling process which takes place upstream of the definition of the problem itself, or in any way, independently of the will of the learner.
It is indeed formalized as an empirical process in accordance with an unknown distribution.
On the contrary, in an inverse problem the discretization usually takes place downstream of the problem: for example, in the case of an industrial device, it can be defined during the design phase or determined even later, after the signal acquisition, as a variable to be optimized in the inversion process.
The \cref{discretization-schemes} summarizes the main sampling schemes corresponding to different applications.
\begin{table}[!h]
\caption{Discretization schemes of a reproducing kernel approximation problem}\label{discretization-schemes}
\begin{center}
\begin{tabular}{c c c}
\toprule
& \multicolumn{2}{c}{sampling $S^{(n)}_{\bar x, \vartheta}$ 
}
\\
\cmidrule(r){2-3}
& $\rho(\cdot|\cdot)$ and $\nu$ unknown & $\bar x$ given and $\rho(\cdot|\cdot)$ known\\
\midrule
direct & learning & interpolation \\
inverse & inverse learning & discrete inverse problems\\
\bottomrule
\end{tabular}
\end{center}
\end{table}

We remark that for learning problems this formulation differs from the classical one where the samples are given without any discretization process.
In the classical formulation the crucial hypothesis is that the samples are drawn independently and identically distributed according to a distribution $\rho(\cdot,\cdot)$ and there is no need to introduce from the beginning $\nu$ and $\rho(\cdot|\cdot)$, but these last two distributions are the result of the factorization of $\rho$. 
Moreover, $y$ is introduced after $\rho$, it depends on the choice of $V$ and represents the parameter of $\rho$ which one wants to learn.

\subsection{Learning from examples}\label{section learning}
We introduce the supervised learning problem in the standard way to highlight the link with our formulation. 
We suppose to know
a finite number of samples 
\begin{equation}
\label{stochastic_samples}
\mathcal{Z}_n
:=  
\{(X_1,Y_1),\dots,(X_n,Y_n)\} ~.
\end{equation}
Such samples are drawn independently identically distributed according to a given (but unknown) probability distribution $\rho$
on $\mathcal{Z} = \mathcal{X}\times \mathcal{Y}$ where $\mathcal{X}\subseteq \mathbb{R}^d$, with $d>0$, and $\mathcal{Y}\subseteq \mathbb{R}$. $\mathcal{X}$ and $\mathcal{Y}$ can be assumed to be compact spaces and $\rho$ admits the following factorization
\begin{equation}
\label{rho} 
\rho(X,Y) = \rho(Y|X) \nu(X)
\end{equation} 
where $\nu$ is the marginal distribution on $\mathcal{X}$ and $\rho( \cdot | X=x)$ is the conditional distribution on $\mathcal{Y}$ for almost $x\in\mathcal{X}$.
Given a measurable function $g$ the ability of $g$ to describe the distribution $\rho$ is measured by the expected risk defined as
\begin{equation}
\label{R_definition}
R_{\rho}(g) = \int_{\mathcal{X}\times\mathcal{Y}} V(Y, g(X)) ~ d\rho(X,Y) ~.
\end{equation}
We remark that thanks to the hypothesis \cref{rho}
$y$,
defined as $y(x)=F_V(\rho(\cdot|X=x))$, is the minimizer of the expected risk \cref{R_definition} (over all measurable functions), i.e. it can be seen as an ideal estimator of the unknown distribution $\rho$ .
However only the set $\mathcal{Z}_n$ is available and therefore 
learning is performed by minimizing over the RKHS $\mathcal{H}_K$ the empirical risk given by
\begin{equation}
\label{ERM-g}
R_{\mathcal{Z}_n}(g)=\frac{1}{n}\sum_{i=1}^n V(Y_i,g(X_i)) ~.
\end{equation}

Therefore the problem \cref{approx_problem} reduces to
\begin{equation}
\label{hat g^(n)_R}
\hat{g}_R^{(n)}:=\arg\min_{g\in\mathcal{H}_K} R_{\mathcal{Z}_n}(g)~.
\end{equation}
From a numerical point of view the solution $\hat{g}^{(n)}_R$ is not stable and therefore, following the approach of Tikhonov regularization, it is useful to introduce a penalty term in order to stabilize the solution. 
Therefore, the regularized problem 
is the following:
\begin{equation}
\label{approx_probl_sampling_noisy}
\hat{g}^{(n)}_{R,\lambda}:=\arg\min_{g\in\mathcal{H}_K} R_{\mathcal{Z}_n}(g)+\lambda\psi(\| g\|_{\mathcal{H}_K})~,
\end{equation}
where $\lambda$ is the regularization parameter.
This is the classical formulation of statistical learning theory, 
in which $\mathcal{X}$ and $\mathcal{Y}$ represent the input and the output space, respectively, and the aim is to find a function $g$ such that $g(x)$ is a good estimate of the output when a new input $x$ is given.

The result of this construction can be obtained by \cref{sampling operator,general-problem} by taking unknown $\vartheta$ and $\nu$, i.e the samples $\mathcal{Z}_n$ can be seen as the result of the action of the sampling operator $S_{\nu,\vartheta}^{(n)}$, where $X_1,\ldots,X_n$ are drawn from the distribution $\nu$, the samples $Y_1,\ldots,Y_n$ are drawn from $\rho(Y | X)$ and the the factorization in $\cref{rho}$ applies.
Moreover, we note that this formulation takes into account the 
inverse statistical learning problem by considering $Af$ instead of $g$ as given in equation \cref{Af=y}
\cite{blanchard2017optimal,rastogi_optimal_2017}.

\subsection{Discrete inverse problems}
In this paragraph we introduce discrete inverse problems with a deterministic discretization scheme \cite{engl_regularization_1996,groetsch_inverse_1993}.
We suppose to know a set of $n$ samples of the infinite dimensional data $y$ (or of a noisy version $y^{\delta}$) computed in the points $x_1,\ldots,x_n$. 
This assumption can be formalized by means of the sampling operator $S^{(n)}_{\bar x}$ which yields the set of samples
\begin{equation}
\label{deterministic_samples}
\mathcal{Z}_n
:= 
\{(x_1,y_1),\dots,(x_n,y_n)\}
\end{equation}
where  $y_1:=y(x_1),\dots,y_n:=y(x_n)$.
\ifdefined\ENGL
LEGAME ENGL: We give a generalization of the convergence result obtained in Theorem 3.24 [Engl] about the convergence of the solution of the projection method, called dual least-square method to the classical generalized solution of inverse problem, that is in the context of least square solutions. In our context we use our generalization of the least square solution of minimal norm. In fact the generalized least square solution is the the solution defined in eq. (\cref{fdaggerR}), when $R_y(Af)=\| y-Af\|^2$. We put in the following more general setting. 
\fi
In this case, the functional \cref{general-problem} takes the form
\begin{equation}
\label{R}
R_y(Af)=\int_{\mathcal{X}} V(y(x),Af(x)) dx~.
\end{equation}
Obviously when $V(y(x),Af(x))=(y(x)-Af(x))^2$ the problem reduces to the least squares minimization with $R_y(Af)=\| y-Af\|^2$.
When the set of samples $\mathcal{Z}_n$ is available 
we minimize the functional
\begin{equation}\label{R_n}
R_{\mathcal{Z}_n}(Af)=\frac{1}{n}\sum_{i=1}^{n} V(y_i,(Af)(x_i))~.
\end{equation}
In this way we estimate the solution of the following discretized inverse problem 
\begin{equation}
\label{discretized inv probl}
y_i = (Af)(x_i) \quad i\in\{1,\dots,n\}.
\end{equation}
Analogously to the procedure followed in \cref{section Infinite dimensional setting}, 
we 
introduce the $R_{\mathcal{Z}_n}$-generalized solution as follows: $(\hat{f}^{(n)}_R)^{\dagger}$ is the minimum norm solution of the problem
\begin{equation}\label{discretized}
\arg\min_{f\in\mathcal{H}_1} R_{\mathcal{Z}_n}(Af).
\end{equation}
If we define $\mathcal{S}_{A,R_{\mathcal{Z}_n}}$ the set of solution of \cref{discretized}, we have that
\begin{equation}\label{fdaggerRdiscretized}
(\hat{f}^{(n)}_R)^{\dagger}=\arg\min_{f\in\mathcal{S}_{A,R_{\mathcal{Z}_n}}} \| f\|_{\mathcal{H}_1}.
\end{equation}
As we noticed in the previous section, it is preferable to regularize the $R_{\mathcal{Z}_n}$-generalized solution for a stability issue, being the data $y_i$ usually corrupted by noise. 
Therefore, in general we solve the Tikhonov regularization problem, i.e.
\begin{equation}
\label{inv_probl_sampling_noisy}
\hat{f}^{(n)}_{R,\lambda}:= \arg\min_{f\in\mathcal{H}_1} R_{\mathcal{Z}_n}(Af)+\lambda \psi(\| f\|_{\mathcal{H}_1}), 
\end{equation}
where $\lambda>0$ is the regularization parameter.

\begin{remark}
Maximum likelihood approach \cite{bertero_introduction_1998,vogel2002computational,kaipio2006statistical,tarantola2005inverse}.
We consider the discretized inverse problem \cref{discretized inv probl}, 
where 
$x_1,\ldots,x_n$ are $n$ points deterministically identified, and for each $i \in \{1,\ldots,n\}$ 
we know the sample $Y_i$
from a given probability distribution $\rho(Y|X=x_i)$. 
The main difference with respect to the learning framework is that here the probability distribution $\rho(\cdot|\cdot)$ is known 
and the quantity to be determined is the parameter $f$ which characterizes the distribution $\rho(\cdot|\cdot)$.
For this reason we denote the distribution $\rho(\cdot|\cdot)$ with $\rho_{Af}(\cdot|\cdot)$ to highlight that it depends on the parameter $Af$. 
This approach can be formalized by means of the sampling operator $S^{(n)}_{\bar x,\vartheta}$ which yields the sample set
\begin{equation}
\label{semi_stochastic_samples}
\mathcal{Z}_n
:= 
\{(x_1,Y_1),\dots,(x_n,Y_n)\} ~,
\end{equation}
where $\vartheta(Af(x)):= \rho_{Af}(\cdot|X=x)$. In the maximum likelihood approach the choice of $V$ is such that
\begin{equation}
V(Y,Af(x)) ~ d\rho_{Af}(Y|X=x) = - \log \rho_{Af}(Y|X=x) ~ dY ~.
\end{equation}
With this choice equation \cref{general-discrete-problem} takes the form
\begin{equation}
\label{ML}
R_{\mathcal{Z}_n}(g) = \frac{1}{n} \sum_{i=1}^n -\log(\rho_{Af}(Y_i|x_i)) ~,
\end{equation}
which corresponds to the negative-log formulation of the maximum likelihood approach.
\end{remark}

The general discrete minimization problem \cref{general-discrete-problem} depends on the set of points $\mathcal Z_n$ but not on their statistical or deterministic origin, i.e. it does not depend on the specific choice of $\bar x$ and $\vartheta$.
For this reason, 
we use the same notation for the solutions of the discretized problems $\hat{g}^{(n)}_R$, $(\hat{f}^{(n)}_R)^{\dagger}$, $\hat{g}^{(n)}_{R,\lambda}$ and $\hat{f}^{(n)}_{R,\lambda}$ regardless the nature of samples $\mathcal{Z}_n$.
In this respect, we conclude this subsection by giving the following
\begin{corollary}\label{corollary correspondence solution discretization}
Given $\mathcal{Z}_n$ a set of samples,  
under \cref{assumption1} and by assuming that $\forall$ $x,x'\in\mathcal{X}$ $K(x,x')=<\phi_x,\phi_{x'}>_{\mathcal{H}_1}$
we have that
\begin{equation}\label{correspondence discretization}
(\hat{f}^{(n)}_R)^{\dagger}
= \tilde{A}^{-1} 
\hat{g}^{(n)}_R.
\qquad \text{and} \qquad 
\hat{f}^{(n)}_{R,\lambda}
= \tilde{A}^{-1} 
\hat{g}^{(n)}_{R,\lambda}
\end{equation}
Furthermore, the solutions $\hat{g}^{(n)}_R$ and $\hat{g}^{(n)}_{R,\lambda}$  of problems \cref{hat g^(n)_R} and \cref{approx_probl_sampling_noisy} correspond to the set of solutions $\{U (\hat{f}^{(n)}_R)^{\dagger} ~|~ U\in\mathcal{U} \}$ and $\{U \hat{f}^{(n)}_{R,\lambda} ~|~ U\in\mathcal{U}\}$, respectively, where we remind that $\mathcal{U}$ is the set of unitary operators on $\mathcal{H}_1$.
\end{corollary}
This result is valid for any choice of $\bar x$ and $\vartheta$, i.e. independently of the discretization scheme.
The proof is omitted since it is a straightforward application of \cref{correspondence} and \cref{theorem methods}. 

\subsection{Convergence}\label{section Convergence}

In this paragraph we discuss the convergence of the empirical functional \cref{general-discrete-problem}
to the ideal one \cref{general-problem} and the convergence of their respective minimizers.
In the case $\mathcal Z_n$ is randomly drawn the convergence is defined in terms of probabilities and the conditions are well established \cite{vapnik2013nature,norkin2009convergence}.
However, if $\mathcal Z_n$ is assumed to be generated in a deterministic manner, the convergence is defined in terms of norms and the theoretical tools for proving the convergence are slightly different.
Indeed, whereas in the statistical framework convergence is a consequence of a straightforward application of 
the argmax continuous theorem \cite{van1996weak}, we show that in the deterministic framework we need a result relying on the notion of $\Gamma$-convergence \cite{braides2006handbook}. 

\subsubsection{Statistical setting} 
\label{statistical setting}

We recall a classical theorem ensuring the consistency of a sequence of $\arg\max$-estimators in an {\it argmin} version suitable for our framework \cite{van1996weak}.
Let $(H,d)$ be a metric space and $(F_n)$ be a sequence of random functions over $H$ given a probability distribution $\nu$. 
\begin{theoremmm}\label{argmax theorem}
(Argmax continuous theorem).
Let us suppose
\begin{equation}
\label{hp1 argmax theorem}
\sup_{h\in H} |F_n (h)-F(h)|\to^{\mathbb{P}} 0 ~,
\end{equation}
where $F$ is a fixed function over $H$ and for each $\epsilon>0$
\begin{equation}\label{hp2 argmax theorem}
\inf_{h\in H : d( h , h^*) \ge\epsilon} F(h)>F(h^*),
\end{equation}
where $h^*$ is the minimizer of $F$. Moreover, if $F_n( h^{(n)} ) \le F_n (h^*) + o_{\mathbb{P}}(1)$, we have
\begin{eqnarray}\label{convergence g probability}
h^{(n)} \to^{\mathbb{P}} h^*
\end{eqnarray}where $h^{(n)}$ is the minimizer of $F_n$.
\end{theoremmm}
Whereas 
the second hypothesis 
is a property of the limit function $F$ at its minimum point $h^*$, which is assured
when $F$ is strictly convex, coercive and lower semi-continuous,
the first hypothesis \cref{hp1 argmax theorem} requires the uniform convergence of $(F_n)$.
When $F_n$ takes the form of the empirical risk (equation \cref{general-discrete-problem}) 
and $F$ is given by equation \cref{general-problem}
the condition \cref{hp1 argmax theorem} is satisfied if $H$ is a uniform Glivenko-Cantelli class (uGC), provided that $V$ has some Lipschitz property \cite{mukherjee2006learning}. 
Then, we have the following
\begin{corollary}\label{probability conv of g}
Let $\mathcal{H}_K$ be uGC. Let $R_{y}$ be defined in \cref{general-problem} and let $V$ be a loss function as in \cref{discrete-data} with the additional Lipschitz property described in \cite{mukherjee2006learning}. Assume that $V$ satisfies the following coercivity property: for each sequence $(g_k)\subseteq\mathcal{H}_K$ such that $\Vert g_k\Vert_{\mathcal{H}_K}\to\infty$, as $k\to\infty$ then $V(Y,g_k(X))\to\infty$, as $k\to\infty$, for each $Y\in\mathcal{Y}$ and $X\in\mathcal{X}$.
Then as $n\to+\infty$,
\begin{equation}
\label{convergence f probability}
\quad\hat{g}^{(n)}_R \longrightarrow^{\mathbb{P}} g_{R_{y}}  \quad \text{and} \quad
(\hat{f}^{(n)}_R)^{\dagger} \longrightarrow^{\mathbb{P}} f^{\dagger}_{R_{y}}, 
\end{equation}
where $\hat{g}^{(n)}_R$ is defined in equation \cref{hat g^(n)_R}, $g_{R_y}$ 
is the minimizer of $R_{y}$ over $\mathcal{H}_K$, $(\hat{f}^{(n)}_R)^{\dagger}$ is defined in equation \cref{fdaggerRdiscretized} and $f^{\dagger}_{R_{y}}$ is the $R_{y}$-generalized solution in according to the definition in \cref{fdaggerR}, 
respectively. 
\end{corollary}
\begin{proof}
Let us take $F_n :=R_{\mathcal{Z}_n}$ (where $R_{\mathcal{Z}_n}$ is defined in \cref{general-discrete-problem})
and $F:=R_{y}$ in \cref{argmax theorem}. 
Condition \cref{hp1 argmax theorem} is verified for the uGC hypothesis on $\mathcal{H}_K$. Condition \cref{hp2 argmax theorem} is verified thanks to the hypothesis of uniqueness of the minimizer of $R_{y}$. Moreover, the sequence $\hat{g}_R^{(n)}$ satisfies $R_{\mathcal{Z}_n}(\hat{g}_R^{(n)})\le R_{\mathcal{Z}_n}(g_{R_{y}})+o_{\mathbb{P}}(1)$ as  $\hat{g}_R^{(n)}$ is the minimizer of $R_{\mathcal{Z}_n}$.
Using the equivalence of learning and inverse problems, we have the following equalities 
\begin{equation}
\label{norm equality}
\| \hat{g}^{(n)}_{R}-g_{R_{y}}\|_{\mathcal{H}_K}=\| A(\hat{f}^{(n)}_R)^{\dagger}-Af^{\dagger}_{R_{y}}\|_{\mathcal{H}_K}=\| (\hat{f}^{(n)}_R)^{\dagger}-f^{\dagger}_{R_{y}}\|_{\mathcal{H}_1} ~.
\end{equation}
This completes the proof.
\end{proof}
\begin{remark}
The same convergence result of \cref{probability conv of g} applies for
Tikhonov type regularized solutions, i.e. fixed $\lambda>0$ we have that
$\hat{g}^{(n)}_{R,\lambda}$ and $\hat{f}^{(n)}_{R,\lambda}$ 
converge in probability to $\hat{g}_{R_{y},\lambda}$, and $\hat{f}_{R_{y},\lambda}$, 
respectively.
\end{remark}
%


\ifdefined\INVERTIBILITY
It is worth observing that the results in eq. (\cref{convergence g probability}) and (\cref{convergence f probability}) are convergence results in probability of the solution of the discretized problem to the solution of the corresponded infinite dimensional problem. The convergence result in eq. (\cref{convergence g probability}) holds under the learnability condition. Therefore it could be interesting translating such concept in the inverse problems setting as an 'invertibility'/'approximatibility' concept. The more recent results has been established a characterization of the concept of learnability through the concept of stability of a learning algorithm [ ]. Stability is used as an alternative to design consistent learning algorithms and it is a concept that has a crucial role in the theory of regularization of ill-posed problem. In fact such concept has emerged in learning to be used to interpret supervised learning as an ill-posed inverse problem.
In this work we only suggest a concept of 'invertibility'/'approximatibility' in a deterministic setting in order to have the convergence of the solution of a discretized inverse problem to the solution of the associated inverse problem in the infinite dimensional setting (presented in the section ..). 
We give the concept 'invertibility'/'approximatibility' using the notion of $\Gamma$-convergence and we obtain a convergence result by using the fundamental theorem of $\Gamma$-convergence as alternative of the argmax theorem. 
\fi

\subsubsection{Deterministic setting}
\label{deterministic setting}

The convergence in the deterministic case needs
the use of the fundamental theorem of $\Gamma$-convergence \cite{braides2006handbook}. First, we recall the $\Gamma$-convergence definition for a given sequence $(F_n)$ of functions on a metric space $(H,d)$ with respect to the distance $d$. 
\begin{definition}
The sequence $(F_{n})$ $\Gamma$-converges in $H$ to a fixed function $F$ if for all $h\in H$ the {\it lim inf inequality} holds, i.e.
for all sequence $h_n$ such that $d(h_n,h)\to 0$, as $n\to+\infty$
\begin{equation}
F(h)\le \lim\inf_n F_{n}(h_n)
\end{equation}
and the {\it lim sup inequality} holds, i.e.
there exists a sequence $h_{n}$ such that $d(h_n,h)\to 0$, as $n\to+\infty$ such that
\begin{eqnarray}
F(h)\ge \lim\sup_n F_{n}(h_n).
\end{eqnarray}
\end{definition}
In order to prove the $\Gamma$-convergence of a sequence we use the following characterization of the equi-coerciveness of a sequence \cite{maso_introduction_1993}. 
\begin{lemma}
$(F_n)$ is an equi-coercive sequence $\iff$
there exists a lower semicontinuous coercive function $G$ such that $F_n\ge G$ on $H$, for each $n\in\mathbb{N}$. 
\end{lemma}
We also exploit the following result which is a consequence of the fundamental theorem of $\Gamma$-convergence (see \cite{braides_local_2012} for details).
\begin{proposition}\label{gamma conv theorem}
Let $(F_n)$ be an equi-coercive sequence $\Gamma$-converging to $F$. Let $h_n$ be a minimizer of $F_n$, and we assume $F$ admits a unique point of minimum $h$. Then $h_n\to h$, as $n\to+\infty$, i.e. $d(h_n,h)\to 0$, as $n\to+\infty$.
\end{proposition}

We now prove the convergence of the minimizer of $R_{\mathcal{Z}_n}$ to the one of $R_y$ over $\mathcal{H}_K$, where $R_y$ is defined in equation \cref{R} and $V$ is strictly convex, Lipschitz continuous with respect to the second variable and coercive in the sense of the definition given in \cref{probability conv of g}. 
\begin{proposition}\label{lem gamma-conv}
Let $x_1,\dots,x_n\in\mathcal{X}$ such that the sequence of points $(x_n)$ is dense in $\mathcal{X}$. 
Let $R_{\mathcal{Z}_n}$ be defined in equation \cref{R_n} and $R_y$ be defined in equation \cref{R}, with $V$ a Lipschitz continuous with respect to the second variable and coercive 
function.
Then the sequence $(R_{\mathcal{Z}_n})$ is an equi-coercive sequence and it $\Gamma$-converges to $R_y$.
\end{proposition}
\begin{proof}
To prove the equi-coerciveness of the sequence $(R_{\mathcal{Z}_n})$, it is sufficient to observe that $R_{\mathcal{Z}_n}\ge R_{\mathcal{Z}_1}$ for all $n\in\mathbb{N}$ where $R_{\mathcal{Z}_1}(g)=V(y_1,g(x_1))$ and then $R_{\mathcal{Z}_1}$ is coercive and continuous for the hypothesis on $V$.
Now we prove that $(R_{\mathcal{Z}_n})$ $\Gamma$-converges to $R_y$. 
Without loss of generality we assume $\mathcal{X} = [0,1]^d$.
Let $g\in\mathcal{H}_K$ and let $(g_n)$ be a sequence converges to $g$, i.e. $\| g_n-g\|_{\mathcal{H}_K}\to 0$, then we have the following inequality
\begin{equation}
\label{dis}
\left|R_{\mathcal{Z}_n}(g_n)-R_y(g)\right|\le\left|R_{\mathcal{Z}_n}(g)-R_y(g)\right|+\left|R_{\mathcal{Z}_n}(g_n)-R_{\mathcal{Z}_n}(g)\right|.
\end{equation}
The first term in the r.h.s in eq. \cref{dis} converges to 0 as $n\to+\infty$ for the definition of the Riemann integral and for the density of the points $x_i$ in $\mathcal{X}$.
Now we prove that the second term in the r.h.s in eq. \cref{dis} converges to 0.
Under the \cref{assumption1} we have that $\| K_{x_i}\|_{\mathcal{H}_K}\le c,$ $\forall$ $x_i$, where $c$ is a fixed constant. By using the Lipschitz continuity 
of $V$ and the reproducing property of $K$ we have the following inequalities
\begin{eqnarray}
\left|R_{\mathcal{Z}_n}(g_n)-R_{\mathcal{Z}_n}(g)\right| & \le &\frac{1}{n}\sum_{i=1}^{n} \left|V(g_n(x_i),y_i)-V(g(x_i),y_i)\right| \\ \nonumber
&\le &\frac{1}{n}\sum_{i=1}^{n} \sigma |g_n(x_i)-g(x_i)| 
\le c \sigma \| g_n-g\|_{\mathcal{H}_K},
\end{eqnarray}
where $\sigma$ is the Lipschitz constant of $V$.
Therefore, for each sequence $(g_n)_n$ converging to $g$ there exists $\lim_{n\to+\infty} R_{\mathcal{Z}_n}(g_n)=R(g)$. Then $(R_{\mathcal{Z}_n})$ $\Gamma$-converges to $R$.
\end{proof}
\begin{corollary}
\label{corollary g_n}
Under the assumptions of \cref{lem gamma-conv} and requiring that $V$ is strictly convex with respect to the second variable we consider $\hat{g}^{(n)}_R$, $g_{R_y}$,  $(\hat{f}^{(n)}_R)^{\dagger}$ and $f_{R_y}^{\dagger}$,  defined in equations \cref{hat g^(n)_R}, \cref{approx_problem},  \cref{fdaggerRdiscretized} and \cref{fdaggerR}, respectively. Then, as $n \to \infty$
\begin{eqnarray}
\hat{g}^{(n)}_{R}\to g_{R_y}, \quad \text{and} \quad
(\hat{f}^{(n)}_R)^{\dagger} \longrightarrow f^{\dagger}_{R_y},  
\end{eqnarray}
where the convergence is uniform in $\mathcal{H}_K$ and $\mathcal{H}_1$, respectively.
\end{corollary}
\begin{proof}
The convergence in $\mathcal{H}_K$ follows from \cref{gamma conv theorem} and \cref{lem gamma-conv}, 
by observing that $R_y$ 
admits a unique minimizer. 
The convergence in $\mathcal{H}_1$ follows from the equality in \cref{norm equality}. 
\end{proof}
\begin{remark}
The same convergence result of \cref{corollary g_n} applies for Tikhonov type regularized solutions, i.e. 
$\hat{g}^{(n)}_{R_y,\lambda}$ and  $\hat{f}^{(n)}_{R_y,\lambda}$ converge to $\hat{g}_{R_y,\lambda}$ and $\hat{f}_{R,\lambda}$, respectively. Such a result follows from the fact that $(R_{\mathcal{Z}_n}+\lambda \psi(\|\cdot\|_{\mathcal{H}_K}))$ is equi-coercive and $\Gamma$-converges to $R_y+\lambda\psi(\|\cdot\|_{\mathcal{H}_K})$ which is a straightforward consequence of \cref{lem gamma-conv} and the fact that $\lambda\psi(\Vert\cdot\Vert_{\mathcal{H}_K})$ is continuous (see \cite{braides2006handbook}). 
\end{remark}

Finally, as the convergence property of the $R_{\mathcal{Z}_n}$-generalized solution 
holds regardless the discretization scheme 
we can summarize functionals, solutions, 
convergence and discretization with the commutative diagrams shown in \cref{cd}.

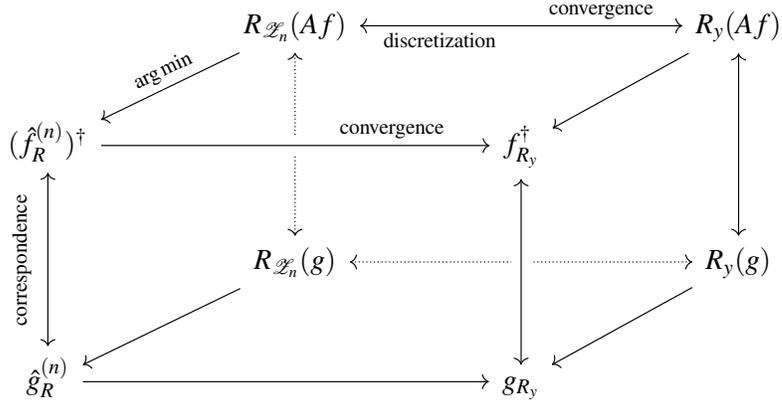
\begin{figure}[h!]
\begin{center}
\begin{tikzcd}[back line/.style={densely dotted}, row sep=2.em, column sep=4.5em]
& R_{\mathcal{Z}_n}(Af) 
\ar[]{dl}
[swap,sloped,near start]{\arg\min} 
\ar[leftrightarrow]{rr}
[near end]{\text{convergence}}
[near start,below]{\text{discretization}} 
\ar[back line,leftrightarrow]{dd}[anchor=center,rotate=90,yshift=2ex]{}
& & R_{y}(Af) 
\ar[leftrightarrow]{dd}
\ar{dl}[swap,sloped,near start]{}\\
(\hat{f}^{(n)}_R)^{\dagger}
\ar[crossing over,rightarrow]{rr}{} 
[near end]{\text{convergence}}
\ar[leftrightarrow]{dd}[anchor=center,rotate=90,yshift=2ex]{\text{correspondence}} 
& & f^{\dagger}_{R_{y}} 
\ar[leftrightarrow]{dd}[anchor=center,rotate=90,yshift=2ex]{} \\
& R_{\mathcal{Z}_n}(g) 
\ar[back line,crossing over,leftrightarrow]{rr}
\ar{dl}[swap,sloped,near start]{}
& & R_{y}(g) 
\ar{dl}[swap,sloped,near start]{}\\
\hat{g}_R^{(n)}
\ar[crossing over,rightarrow]{rr} 
& & g_{R_{y}} 
\ar[crossing over, leftrightarrow]{uu}
\end{tikzcd}
\caption{A summary of the discretization and convergence results applied to the approximation problems in a RKHS. Arrows indicate: from left to right convergence processes; from right to left 
discretization processes in the rear panel; from rear to front optimization processes; from top to bottom (and viceversa) the correspondence between inverse and direct problems.}\label{cd}
\end{center}
\end{figure}
The vertexes of the rear side of the cube represent the four minimizing functionals and the vertexes of the front side represent the corresponding solutions. The empirical and ideal cases are shown on the left and right sides, respectively. The arrows from left to right represent the convergence, while the arrows from right to left, on the rear side, represent the discretization. 
The arrows from rear to front show the minimizing process. In particular,
along horizontal arrows of the front side of the cube we show the convergence of the empirical solutions to the ideal ones (\cref{probability conv of g,corollary g_n});
along vertical arrows we show the correspondence between solutions of approximation problems in a RKHS and inverse problems (\cref{correspondence,corollary correspondence solution discretization}). 
\subsection{Application of the representer theorem}
The representer theorem and its generalizations prove that the solution of problem \cref{approx_probl_sampling_noisy} belongs to a finite dimensional subspace of $\mathcal{H}_K$ \cite{scholkopf2001generalized}. 
Under the \cref{assumption1} on the linear operator $A:\mathcal{H}_1\to\mathcal{H}_2$, %
let 
\begin{equation}
\label{H_K^n}
\mathcal{H}_{K}^{(n)}:=span\{K_{x_1},\dots,K_{x_n}\}
\end{equation}
and 
\begin{equation}
\label{H_1^n}
\mathcal{H}_1^{(n)}:=span\{\phi_{x_1},\dots,\phi_{x_n}\},
\end{equation}
be two finite dimensional subspaces $\mathcal{H}_K^{(n)} \subset \mathcal{H}_K$ and $\mathcal{H}_1^{(n)} \subset \mathcal{H}_1$, where $\phi$ and $K$ are related by the equation \cref{Kernel}.
Under the aforementioned conditions on the loss function $V$ and $\psi$ (on which depends the penalty term), in the statistical learning setting the representer theorem allows us to write
\begin{equation}
\label{repr_g_noisy}
\hat{g}^{(n)}_{R,\lambda}=\sum_{i=1}^{n} \beta_i K_{x_i},
\end{equation}
where $\beta_i\in\mathbb{R}$ for all $i\in\{1,\dots,n\}$ are appropriate coefficients. 
Thus, 
the problem \cref{approx_probl_sampling_noisy} can be re-formulated as follows
\begin{equation}\label{prob g}
\hat{g}_{R,\lambda}^{(n)}:=\arg\min_{g\in\mathcal{H}^{(n)}_K} R_{\mathcal{Z}_n}(g)+\lambda\psi(\| g\|_{\mathcal{H}_K}),
\end{equation}
where the optimization is performed on the finite dimensional subspace $\mathcal{H}^{(n)}_K$. 
Clearly, \cref{corollary correspondence solution discretization} 
can be exploited to provide a representer theorem 
for $\hat{f}^{(n)}_{R,\lambda}$. 
\begin{proposition}
\label{representer-theorem-f}
The solution of the problem
\cref{inv_probl_sampling_noisy} admits the following representation
\begin{eqnarray}
\hat{f}^{(n)}_{R,\lambda}=\sum_{i=1}^{n} \beta_i \phi_{x_i},
\end{eqnarray}
where $\beta_i\in\mathbb{R}$, for all $i\in\{1,\dots,n\}$ are the same coefficients of equation \cref{repr_g_noisy}.
Finally the problem 
\cref{inv_probl_sampling_noisy} can be re-formulated as follows
\begin{equation}\label{prob f}
\hat{f}_{R,\lambda}^{(n)}:=\arg\min_{f\in\mathcal{H}^{(n)}_1} R_{\mathcal{Z}_n}(A f)+\lambda\psi(\| f\|_{\mathcal{H}_1}),
\end{equation}
where $\mathcal{H}_1^{(n)}$ is defined in equation \cref{H_1^n}.
\end{proposition}
The major consequence of this result is that it is sufficient to determine coefficients $\{ \beta_j \}_{j=1}^{n}$ in order to solve both problems \cref{prob g,prob f}. 
For the sake of completeness, we report the explicitly computation of the coefficients $\beta_j$ in the classical Tikhonov regularization case.

\begin{example}
Let us consider the Tikhonov regularization for a linear inverse problem which is known as penalized least square approach in supervised learning. Under the usual assumptions, we write the problem \cref{prob g} as
\begin{equation}\label{tik approx pb}
\hat{g}^{(n)}_{\lambda}=\arg\min_{g\in\mathcal{H}_K} \frac{1}{n}\sum_{i=1}^{n} (y_i-g(x_i))^2+\lambda\| g\|^2_{\mathcal{H}_K} ~,
\end{equation}
and the problem \cref{prob f} as
\begin{equation}\label{tik inv pb}
\hat{f}^{(n)}_{\lambda} =\arg\min_{f\in\mathcal{H}_1} \frac{1}{n}\sum_{i=1}^{n} (y_i-Af(x_i))^2+\lambda\| f\|^2_{\mathcal{H}_1} ~.
\end{equation}
Using the representer theorem and equation \cref{correspondence discretization} the solution of the two problems \cref{tik approx pb} and \cref{tik inv pb} is given by solving the following
\begin{eqnarray}
\hat{\beta}_{\lambda}^{(n)} =\arg\min_{\beta\in\mathbb{R}^n} \frac{1}{n}\| \textbf{y}-\textbf{K}\beta\|^2_2+\lambda \beta^T \textbf{K} \beta,
\end{eqnarray}
where $\textbf{y}$ is the $n$-dimensional vector $\textbf{y}=(y_1,\dots,y_n)^T$, and $\textbf{K}$ is the matrix with entries $\textbf{K}_{ij}:= K(x_i,x_j)$, for each $i,j\in\{1,\dots,n\}$.
In such a case the solution $\hat{\beta}^{(n)}$ is given by
\begin{eqnarray}
\hat{\beta}_{\lambda}^{(n)}=(\textbf{K}+\lambda n I)^{-1}\textbf{y},
\end{eqnarray}
and therefore solutions of problems \cref{tik approx pb} and \cref{tik inv pb} are given respectively by
\begin{eqnarray}
\hat{g}^{(n)}_{\lambda}=\textbf{k}^T(\textbf{K}+\lambda n I)^{-1}\textbf{y},
\end{eqnarray}
where $\textbf{k}=(K_{x_1},\dots,K_{x_n})^T$, and
\begin{eqnarray}
\hat{f}^{(n)}_{\lambda}=\Phi^T(\textbf{K}+\lambda n I)^{-1}\textbf{y},
\end{eqnarray}
where $\Phi=(\phi_{x_1},\dots,\phi_{x_n})^
T$.
\end{example}

Analogously, the solutions $\hat{g}^{(n)}_R$ and $(\hat{f}^{(n)}_R)^{\dagger}$ defined in equations \cref{hat g^(n)_R} and \cref{fdaggerRdiscretized} respectively admit a finite representation. This follows from the fact that $\hat{g}^{(n)}_R$ can be seen as the minimizer of \cref{prob g} with $\psi=0$. Hence, at least a minimizer has a finite representation as $\psi$ is non-decreasing and it is unique as $R_y$ is strictly convex \cite{yu2013characterizing,argyriou2014unifying}.
In the next proposition we give a simple alternative proof of the fact that $\hat{g}^{(n)}_R$ and $(\hat{f}^{(n)}_R)^{\dagger}$ admit a finite representation based on $\Gamma$-convergence.
\begin{proposition}
\label{representer g_n}
Let $R_{\mathcal{Z}_n}$ be defined in equation \cref{R_n} , with $V$ strictly convex, coercive (as the definition in \cref{probability conv of g}) and Lipschitz continuous function. 
The solution $\hat{g}^{(n)}_R$ defined in equation \cref{hat g^(n)_R} admits the following representation
\begin{equation}
\label{representer theorem g_R}
\hat{g}^{(n)}_R=\sum_{j=1}^{n} \alpha_j K_{x_j},
\end{equation}
where $\alpha_j\in\mathbb{R}$, for all $j\in\{1,\dots,n\}$ are appropriate coefficients. 
\end{proposition}
\begin{proof}
Let $\lambda>0$ and let $\psi$ be a continuous convex and strictly increasing real-valued function. 
Fixed $n\in\mathbb{N}$, the sequence $(R_{\mathcal{Z}_n}+\lambda\psi(\|\cdot\|_{\mathcal{H}_K}))_{\lambda}$ satisfies the hypotheses of \cref{gamma conv theorem} and it $\Gamma$-converges to $R_{\mathcal{Z}_n}$ as $\lambda\to 0$. This proves the convergence of minimizers, i.e. $\hat{g}^{(n)}_{R,\lambda}\to\hat{g}^{(n)}_R$ as $\lambda\to 0$, uniformly in $\|\cdot\|_{\mathcal{H}_K}$ for all $n\in\mathbb{N}$, where $\hat{g}^{(n)}_{R,\lambda}$ is defined in equation \cref{ERM-g}.
Moreover, $\hat{g}^{(n)}_{R,\lambda}$ admits the following representation
\begin{eqnarray}
\hat{g}^{(n)}_{R,\lambda} =\sum_{j=1}^{n}\beta^{\lambda}_j K_{x_j},
\end{eqnarray}
where $\beta^{\lambda}_j\in\mathbb{R}$ for all $j\in\{1,\dots,n\}$.
Therefore, $\sum_{j=1}^{n}\beta_j^{\lambda} ~ K_{x_j}$ pointwise converges to $\hat{g}_R^{(n)} \in \mathcal H_K$  as $\lambda\to 0$ and each $\beta_j^{\lambda}$ has to converge to some value $\beta_j^0$.
The limit can be written as $\sum_{j=1}^{n}\beta_j^0~ K_{x_j}$and this shows that $\hat{g}_R^{(n)} \in \mathcal H^{(n)}_K$.
\end{proof}
\begin{corollary}
Under assumptions of \cref{representer g_n}, consider $(\hat{f}^{(n)}_R)^{\dagger}$ defined in equation \cref{fdaggerRdiscretized}. Then $(\hat{f}^{(n)}_R)^{\dagger}$ admits the following representation
\begin{equation}\label{representer theorem f_R}
(\hat{f}^{(n)}_R)^{\dagger}=\sum_{j=1}^{n} \alpha_j \phi_{x_j},
\end{equation}
where $\alpha_j\in\mathbb{R}$, for all $j\in\{1,\dots,n\}$ are the same coefficients in equation \cref{representer theorem g_R}.
\end{corollary}
The proof is a straightforward consequence of \cref{representer g_n} and \cref{corollary correspondence solution discretization}.
\ifdefined\ENGL
LEGAME ENGL: We give a generalization of the convergence result obtained in Theorem 3.24 [Engl] about the convergence of the solution of the projection method, called dual least-square method to the classical generalized solution of inverse problem, that is in the context of least square solutions. In our context we use our generalization of the least square solution of minimal norm. In fact the generalized least square solution is the the solution defined in eq. (\cref{fdaggerR}), when $R_y(Af)=\| y-Af\|^2$. We put in the following more general setting. Let
\begin{eqnarray}\label{R}
R_y(Af)=\int_{\mathcal{X}} V(y,Af(x)) dx,
\end{eqnarray}
where for simplicity $\mathcal{X}$ is a compact set of $\mathbb{R}^p$ and without loss of generality we set $\mathcal{X}=[0,1]^p$ and we assume $V$ continuous. Obviously if $V(y,Af(x))=(y-Af(x))^2$ ricadiamo nel caso della generalized least square solution. Da Engl si ha che la soluzione least square del problema proiettato converge alla generalized least square solution. In our setting we consider the solution of the projected problem as defined the solution of minimum norm which minimizes
\begin{eqnarray}\label{R_n}
R_{n}(y,Af)=\frac{1}{n}\sum_{i=1}^{n} V(y_i,(Af)(x_i)).
\end{eqnarray}
In the case of least square $R_n(y,Af)=\frac{1}{n}\sum_{i=1}^{n} (y_i-(Af)(x_i))^2$, which is the norm in $\mathbb{R}^n$ associated to the inner product of the empirical $L^2$ structure 
\begin{eqnarray}
<y,y'>_{\mathbb{R}^n}:=\frac{1}{n}\sum_{i=1}^{n} y_i y'_i,
\end{eqnarray}
che formalmente si ottiene andando a sostituire al posto della misura di Lebesgue la misura ....(non mi ricordo il nome). Noi generalizziamo, considerando $V$ una loss generica, alla quale chiediamo la continuita'. Diamo il seguente teorema
\begin{lemma}
Let $F_{n,\lambda}$ be the functional defined as following: $y\in\mathcal{H}_2$, $g\in\mathcal{H}_K$
\begin{eqnarray}
F_{n,\lambda}(y,g) = R_{n}(y,g)+\lambda\psi(\| g\|_{\mathcal{H}_K}),
\end{eqnarray}
where $R_n$ is defined in eq. (\cref{R_n}). Let $F_{\lambda}$ be 

\end{lemma}
\begin{theoremmm}
Let $x_1,\dots,x_n\in\mathcal{X}$ such that the set of point $\{x_n\}$ is dense in $\mathcal{X}$. Let $\lambda>0$. Let $\hat{f}_{R,\lambda}$ be the solution of
\begin{eqnarray}\label{f}
\arg\min_{f\in\mathcal{H}_1} R_y(Af)+\lambda\psi(\| f\|_{\mathcal{H}_1}),
\end{eqnarray}
where $R$ is defined in eq. (\cref{R}), 
with $V$ continuous. Let $(\hat{f}^{n}_{R,\lambda})$ be the solution of minimum norm of 
\begin{eqnarray}\label{f_n}
\arg\min_{f\in\mathcal{H}_1} R_n(y,Af)+\lambda\psi(\| f\|_{\mathcal{H}_1}),
\end{eqnarray}
where $R_n$ is defined in eq. (\cref{R_n}). We assume that the sequence $(R_n+\lambda\psi(\| \cdot\|_{\mathcal{H}_1}))_n$ is equicoercive (basta forse scrivere $(R_n)_n$ equicoercive).
Then
\begin{eqnarray}
(\hat{f}^{(n)}_{R,\lambda})\to \hat{f}_{R,\lambda}, \text{ as } n\to+\infty.
\end{eqnarray}
\end{theoremmm}
\begin{proof}
We consider the approximation problems in the RKHS $\mathcal{H}_K$ associated to the problems in eq. $(\cref{f})$ and (\cref{f_n}), respectively (seen section ..). Let $\hat{g}^{(n)}_{R,\lambda}$ be the solution of the problem
\begin{eqnarray}\label{g_n}
\hat{g}^{(n)}_{R,\lambda}=\arg\min_{g\in\mathcal{H}_K} R_n(y,g)+\lambda\psi(\| g\|_{\mathcal{H}_K})
\end{eqnarray}
and let $\hat{g}_{R,\lambda}$ be the solution of the problem
\begin{eqnarray}
\hat{g}_{R,\lambda}=\arg\min_{g\in\mathcal{H}_K} R_y(g)+\lambda\psi(\| g\|_{\mathcal{H}_K}).
\end{eqnarray}
First we prove that $(R_n+\lambda\psi(\| \cdot\|_{\mathcal{H}_K}))_n$ $\Gamma$-converges to $R+\lambda\psi(\| \cdot\|_{\mathcal{H}_K}$. Let $g\in\mathcal{H}_K$ and let $(g_n)_n$ be a sequence converges to $g$, i.e. $\| g_n-g\|_{\mathcal{H}_K}\to 0$ then we observe the following inequality
\begin{eqnarray}
 & & \left|R_n(y,g_n)+\lambda\psi(\| g_n\|_{\mathcal{H}_K})-(R_y(g)+\lambda\psi(\| g\|_{\mathcal{H}_K}))\right|\le \label{dis1}\\
 & &\le\left|R_n(y,g)-R_y(g)\right|+\left|R_n(y,g_n)-R_n(y,g)\right|+ \label{dis2}\\
 & & + \lambda \left| \psi(\| g_n\|_{\mathcal{H}_K})-\psi(\| g\|_{\mathcal{H}_K})\right| \label{dis3}
\end{eqnarray}
The first term in eq. (\cref{dis2}) converges to 0, as $n\to+\infty$ for the definition of the Riemann integral and for the density of the points $x_i$ in $\mathcal{X}$. In hypothesis $\psi$ continuous the term in eq. (\cref{dis3}) converges to 0, as $n\to+\infty$.  Now we prove the second term in eq. (\cref{dis2}) converges to 0.
Under assumption \cref{assumption1} we observe that $\| K_{x_i}\|_{\mathcal{H}_K}\le c,$ $\forall$ $x_i$, where $c$ is a fixed constant. By using the Lipschitz hypothesis of $V$, the reproducing property of the r.k. $K$ we conclude with thanks to the following inequalities
\begin{eqnarray}
\left|R_n(y,g_n)-R_n(y,g)\right| & \le &\frac{1}{n}\sum_{i=1}^{n} \left|V(g_n(x_i),y_i)-V(g(x_i,y_i))\right|\le\\
&\le &\frac{1}{n}\sum_{i=1}^{n} L |g_n(x_i)-g(x_i)|\le\\
& \le & \frac{1}{n}\sum_{i=1}^{n} L \| g_n-g\|_{\mathcal{H}_K}\| K_{x_i}\|_{\mathcal{H}_K}\le L c \| g_n-g\|_{\mathcal{H}_K},
\end{eqnarray}
where $L$ is the Lipschitz constant of $V$.
Therefore for each sequence $(g_n)_n$ converging to $g$ exists $\lim_{n\to+\infty} R_n(y,g_n)+\lambda\psi(\| g_n\|_{\mathcal{H}_K})=R_y(g)+\lambda\psi(\| g\|_{\mathcal{H}_K})$, so $(R_n+\lambda\psi(\|\cdot\|_{\mathcal{H}_K}))_n$ $\Gamma$-converges to $R+\lambda\psi(\|\cdot\|_{\mathcal{H}_K})$. The hypothesis of the fundamental theorem of $\Gamma$-convergence hold and so the minimum of $R_n+\lambda\psi(\|\cdot\|_{\mathcal{H}_K})$ converges to the minimum of $R+\lambda\psi(\|\cdot\|_{\mathcal{H}_K})$.
\fi

\ifdefined\DEBUG
Finally to have the thesis $\hat{f}_{R,\lambda}^{(n)}\to \hat{f}_{R,\lambda}$, as $n\to+\infty$ we use the correspondence of the solution $\hat{g}^{(n)}_{R,\lambda}=A\hat{f}^{(n)}_{R,\lambda}$ and $\hat{g}_{R,\lambda}=A \hat{f}_{R,\lambda}$ and the following equality
\begin{eqnarray}
\| \hat{g}^{(n)}_{R,\lambda}- \hat{g}_{R,\lambda}\|_{\mathcal{H}_K}=\| \hat{f}^{(n)}_{R,\lambda}-\hat{f}_{R,\lambda}\|_{\mathcal{H}_1}.
\end{eqnarray}
Furthermore if $K_{x_1},\dots,K_{x_n}$ are linearly independent and we define with $\mathcal{H}_1^{(n)}:=\tilde{A}^{-1}\mathcal{H}_K^{(n)}=A^*\mathcal{H}_K^{(n)}\subset Ker(A)^{\perp},$ where $A^*$ is the adjoint from $\mathcal{H}_K$ to $\mathcal{H}_1$, we have that $\mathcal{H}_1^{(n)}=span\{\phi_{x_1},\dots,\phi_{x_n}\}$, where $\phi_{x_i}=\tilde{A}^{-1} K_{x_i}= A^* K_{x_i}$. Since $K_{x_i}$ are linearly independent, then $\phi_{x_i}$ are linearly independent. Thanks to the representer theorem
\begin{eqnarray}
\hat{f}^{(n)}_{R,\lambda}=\sum_{j=1}^{n} \alpha_j \phi_{x_j},
\end{eqnarray}
with $\alpha_j\in\mathbb{R}$, i.e. $\hat{f}^{(n)}_{R,\lambda}=T^{(n)}\tilde{f}$, where $T^{(n)}$ is the projector onto $\mathcal{H}_1^{(n)}$ and $\tilde{f}\in Ker(A)^{\perp}$. Since $\mathcal{H}_1^{(n)}$ is an encapsulated sequence of sets $\mathcal{H}_1^{(n)}\subset\mathcal{H}_1^{(n+1)},$ $\forall$ $n\in\mathbb{N}$ and $\bigcup_{n\in\mathbb{N}} \mathcal{H}_1^{(n)}$ is dense in $Ker(A)^{\perp}$ (detto da Engl) then $\hat{f}^{(n)}_{R,\lambda}\to \tilde{f}$. For the uniqueness of limit $\tilde{f}=f^{\dagger}_R$, then $\hat{f}^{(n)}_{R,\lambda}=T^{(n)} \hat{f}_{R,\lambda}$. 
\end{proof}

TEOREMA.
$g_{R,\lambda} \to g_R$ PER $\lambda \to 0$.

Proof: 
\fi


\ifdefined\DEBUG
$g^{n}_{R,\lambda} \to g^{n}_R$ PER $\lambda \to 0$.\\

TEOREMA: $g_n^\dagger \to g^\dagger$.\\
Proof: tripla divisione della norma con cauchy schwartz.
$$\|g^{n}_R - g_R\| \leq \|g^{n}_R - g^{n}_{R,\lambda}\| + \|g^{n}_{R,\lambda} - g_{R,\lambda}\| + \|g_{R,\lambda} - g_R\|$$
Limite su $n$ e $\lambda$ sufficientemente piccolo dato che $\|g^{n}_R - g_R\| = \lim_{\lambda \to 0} \|g^{n}_R - g_R\| $ .\\
\fi


\ifdefined\REGULARIZATION
\subsection{Projection associated with sampling $S_{\nu,\vartheta}^{(n)}$}
\cref{representer-theorem-f} suggests that the solution of the discretized inverse problem \cref{discretized inv probl} could be seen as solution of the infinite dimensional problem \cref{Af=y} projected onto the finite dimensional space $\mathcal H_K^{(n)}$. 
Now we show that the finite dimensional solution can always be interpreted as a projection of the infinite dimensional one.
``
This leads to an extension of the dual least square method as presented in \cite{engl_regularization_1996} to the case of $R$ is different from the least-squares functional.'' 
Consider the sequence of finite dimensional subspaces
\begin{eqnarray}
\mathcal{H}_K^{(1)}\subset\mathcal{H}_K^{(2)}\subset\dots\subset\mathcal{H}_K^{(n)},
\end{eqnarray}
where $\mathcal{H}_K^{(n)}\subset\mathcal{H}_K=\Im(A)$ is defined in equation \cref{H_K^n}. 
We assume that the points $x_i$ are distinct and $K_{x_i}$ are linearly independent of each others. 
Such an assumption is satisfied when $K$ is a strictly positive definite kernel [ ].
Moreover, we introduce the projection $Q_R: \mathcal H_2 \to \mathcal{H}_K$ 
which maps $y \in \mathcal H_2$ in $h \in \mathcal{H}_K$ 
given by
\begin{equation}
h:=\arg\min_{h'\in\mathcal{H}_K} R_y(h').
\end{equation}
We consider that assumptions on $R_y$ (see \cref{Approximation problems in RKHS}) hold true for each $y \in \mathcal H_2$, so that the projection exists and it is uniquely defined.
Moreover, $h = Q_R(h)$ when $h \in \mathcal{H}_K$ and then it is an involution operator, i.e. $Q_R^2 = Q_R$.
Given $y \in \mathcal{H}_2$, we consider a map $s:\mathcal{H}_K \to \mathcal{H}_2$ such that $Q_R \circ s = \mathrm{Id}$ and $s(g_R) = y$, where $g_R$ is defined in \cref{approx_problem}.
From a topological point of view this map is a section of a  fiber bundle $(\mathcal H_2 , \mathcal{H}_K, Q_R , \mathbb{R})$.
For each $h \in \mathcal{H}_K$ we define the operator
\begin{eqnarray}
W_R^{(n)}(h):=\arg\min_{g\in\mathcal{H}_K} R_{\mathcal{Z}_n(h)}(g),
\end{eqnarray}
where
\begin{eqnarray}
\mathcal{Z}_n(h):=\{(x_i,(S^{(n)}_{\bar{x},\vartheta}(s(h)))_i)\}_{i=1}^{n} ~.
\end{eqnarray}
where $S^{(n)}_{\bar{x},\vartheta}$ is the sampling operator defined in \cref{sampling operator}.
Clearly, $W_R^{(n)}$ depends on the choice of the section $s$.
In the case of $R_y$ is the least squares functional we have that the projection is the orthogonal projection for each section $s$, and this construction reduces to the usual regularization by projection defined in \cite{engl_regularization_1996}.
Indeed, by construction of $W_R^{(n)}$ we have
\begin{equation}
\label{projection-g}
\hat{g}^{(n)}_R=W^{(n)}_R(g_{R_{y}}) ~ .
\end{equation}
Thanks  to the representer theorem $W^{(n)}_R$ is a projection onto $\mathcal{H}_K^{(n)}$. (!!!NOTA: sto ragionando su questa parte e per poter dire che $W^{(n)}_R$ e' una proiezione bisogna almeno far vedere che preso un elemento di $h\in\mathcal{H}^{(n)}_K$ allora $W^{(n)}_R(h)=h$. Il problema e' che la definizione e' legata a $y$. Quindi per un elemento generico non funziona.)
This allows us to define the following projection
\begin{equation}
P^{(n)}_R:=\tilde{A}^{-1}W^{(n)}_R A
\end{equation}
for which we have
\begin{equation}
\label{projection-f}
(\hat{f}^{(n)}_R)^{\dagger}=P_R^{(n)} (f^{\dagger}_{R_{y}})
\end{equation}
Therefore, $P^{(n)}_R$ is a projection onto the finite dimensional subspace 
$$
\tilde{A}^{-1}\mathcal{H}_K^{(n)}=A^*\mathcal{H}_K^{(n)}\subset Ker(A)^{\perp},
$$
where $A^*$ is the adjoint from $\mathcal{H}_K$ to $\mathcal{H}_1$ and we have $\tilde{A}^{-1}\mathcal{H}_K^{(n)}=\mathcal{H}_1^{(n)}$ (see equation \cref{H_1^n}).
Since $K_{x_i}$ are linearly independent, then $\phi_{x_i}$ are linearly independent and this is in agreement with $(\hat{f}^{(n)}_R)^{\dagger}
\in\mathcal{H}_1^{(n)}$.


Therefore, the equivalence between problems \cref{approx_problem} and \cref{Af=y} coupled with the representer theorem provides an extension of the concept of regularization by projection 
to the case of $R_y$-generalized solutions.



Furthermore, we remark that the same construction of the projections allows us to extend the regularization by projection also in a statistical setting: when we consider $R_{\rho}$ in place of $R_y$ we obtain 
\begin{equation}
\hat{g}^{(n)}_R = W^{(n)}_R (g_{\rho}) ~ \text{ and } ~ (\hat{f}^{(n)}_R)^{\dagger} = P^{(n)}_R (f^{\dagger}_{R_{\rho}}),
\end{equation}
which is in accordance with the representer theorem, and the convergence of the solutions has to be considered in probability (see \cref{statistical setting}). 
This also extends Theorem 3.4 in \cite{engl_regularization_1996} to the statistical framework.

\begin{remark}
The representer theorem and a construction of projections similar to the above one allow us to write the Tikhonov-type regularized solutions of the discretized problems as projections of the Tikhonov-type regularized solutions of the infinite-dimensional problems. By replacing $R_{(\cdot)}$ with $R_{(\cdot)}+\lambda\psi(\Vert\cdot\Vert)$ and $R_{\mathcal{Z}_n(h)}$ with $R_{\mathcal{Z}_n(h)}+\lambda\psi(\Vert\cdot\Vert)$  (where $R_{(\cdot)}$ allows us to consider $R_y$ or $R_{\rho}$ according to the type of the discretization of the problem) we obtain that $\hat{g}^{(n)}_{R,\lambda}=W^{(n)}_{R+\lambda \psi} (\hat g_{R_{(\cdot)},\lambda})$ and $\hat{f}^{(n)}_{R,\lambda} = P_{R+\lambda\psi}^{(n)} (\hat f_{R_{(\cdot)},\lambda})$.
\end{remark}
\fi

\section{Connection between convergence rates}
The study of the convergence rates is carried out in parallel in different settings as in (inverse) learning problems \cite{bauer2007regularization,blanchard2017optimal}, statistical inverse problems \cite{bissantz2007convergence}, inverse problems with deterministic noise \cite{albani2016optimal,engl_regularization_1996}.
In this section we focus on the convergence rates provided in the statistical learning setting and in the linear inverse problems with deterministic noise. The crucial difference  between these two approaches lies in the independent variable which the error depends on.
Whereas for learning problems the independent variable is the number of examples $n$, for inverse problems it is the noise level $\delta$ dealing with infinite dimensional noisy data. The relation between the optimal rates provided in these two settings under the same source condition is not straightforward \cite{engl_regularization_1996,bauer2007regularization}. It is evident that there is no transformation between $n$
and $\delta$ (independently of the rate) mapping one rate to the other.
The aim of this section is to find  a relation between the two rates and to quantify their difference. To do this we introduce an estimator in the statistical learning setting which is different from the usual one, with
the following properties.
\begin{itemize}
\item The expected error given by this estimator is always larger than the error given by the standard spectral regularized solution provided that a suitable relation between $n$ and $\delta$ holds true. 
Such an inequality allows to convert upper convergence rates depending on $n$ to upper convergence rates depending on $\delta$ and viceversa, lower convergence rates depending on $\delta$ to lower convergence rates depending on $n$ (\cref{subsection conversion rates}).
\item It has the same upper rates of the spectral regularization methods \cite{blanchard2017optimal} (\cref{subsection rate of our estimator}). 
\end{itemize}

\subsection{A link between the number of examples $n$ and the noise level $\delta$}\label{subsection conversion rates}

Let $L^\lambda$ be a linear regularization operator family by varying of $\lambda>0$ from an Hilbert space $\mathcal{H}_2$ to another $\mathcal{H}_1$ \cite{engl_regularization_1996}.
Let $(\mathcal{T},\Theta, \mu)$ be a measure space 
with respect to the measure $\mu$ on $(\mathcal{T},\Theta)$ where $\mathcal{T}$ is a nonempty set and $\Theta$ is a $\sigma$-algebra and $(\mathcal{X}, \Sigma,\nu)$ be a measure space with respect to the measure $\nu$ on $(\mathcal{X},\Sigma)$, where $\mathcal{X}$ is a nonempty set and $\Sigma$ is a $\sigma$-algebra. We assume that $\nu$ is a positive and finite measure. We suppose that $\mathcal{H}_1$ is the $L^2(\mathcal{T},\mu)$ space (the Hilbert space of square integrable functions on $\mathcal{T}$ with respect to the measure $\mu$) and $\mathcal{H}_2$ is the $L^2(\mathcal{X},\nu)$ space (the Hilbert space of square integrable functions on $\mathcal{X}$ with respect to the measure $\nu$).
We assume that $L^{\lambda}$ has the following form
\begin{equation}\label{L_lambda}
L^{\lambda}y=\int_{\mathcal{X}} \ell^{\lambda}_x ~ y(x) ~ d\nu(x)~,
\end{equation}
where $\ell^\lambda_x \in \mathcal H_1$, $\ell_x^{\lambda}(t):=\ell^{\lambda}(x,t)$ and $\ell^{\lambda}(\cdot,t)\in\mathcal{H}_2$ for each $x \in \mathcal X$ and for each $t\in\mathcal{T}$. Thanks to this last assumption the integral in equation \cref{L_lambda} is finite. Moreover, we assume $\sup_{t\in\mathcal{T}}\Vert\ell^{\lambda}(\cdot,t)\Vert_{\mathcal{H}_2}<\infty$. Such an assumption implies that $L^{\lambda}$ is uniformly bounded and then for each $y\in\mathcal{H}_2$ $L^{\lambda} y$ is bounded in supremum norm which assures that $L^{\lambda}y\in\mathcal{H}_1$.   
We denote with $f^{\lambda}$ the regularized solution given by the linear regularization operator $L^{\lambda}$ applied to the noise free data $y$, i.e.
\begin{equation}\label{L lambda y}
f^{\lambda} = L^{\lambda} y,
\end{equation}
and with $f^{\lambda}_{\delta}$ the regularized solution given by the noisy data $y^{\delta}$, i.e.
\begin{equation}\label{solution f lambda delta}
f^{\lambda}_{\delta} = L^{\lambda} y^{\delta},
\end{equation}
when $\Vert y-y^{\delta}\Vert\le\delta$.
We introduce the following estimator computed from a set of discrete data as follows 
\begin{equation}
\label{discretized-linear-regularization}
\hat{f}^{\lambda}_n = L^{\lambda}_{\textbf{x}} \textbf{y} = \frac{1}{n} \sum_{i=1}^n \ell^\lambda_{X_i} ~ Y_i
\end{equation}
where $\textbf{x} = (X_1,\ldots,X_n)$ and $\textbf{y}=(Y_1,\dots,Y_n)$ denote 
the samples.

Convergence rates are usually studied for linear regularization methods based on spectral theory. Now we introduce the standard regularization operator used in the spectral theory in our notation. We denote with $s_{\lambda}$ the regularization function. Then the regularized solution $f^{\lambda}$ is given by
\begin{equation}
\label{f lambda spectral theory}
f^{\lambda} = s_{\lambda}(A^*A) A^* y,
\end{equation}
where $A^*$ is the adjoint operator of $A$. With straightforward computations the solution \cref{f lambda spectral theory} can be re-written in the form \cref{L lambda y} by setting
\begin{equation}\label{ell in sp}
\ell_x^{\lambda} = s_{\lambda}(A^*A) \phi_{x},
\end{equation}
with $x\in\mathcal{X}$. We remark that the required hypotheses on $\ell_x^{\lambda}$, 
are satisfied under the assumption $\sup_{t\in\mathcal{T}}\Vert\phi(\cdot,t)\Vert_{\mathcal{H}_2}<\infty$. 
Furthermore, the estimator defined in \cref{discretized-linear-regularization} takes the following form
\begin{equation}\label{hat f}
\hat{f}^{\lambda}_n = s_{\lambda}(A^*A)A^*_{\textbf{x}}\textbf{y},
\end{equation}
where $A_{\textbf{x}}$ is the sampling operator 
associated to the set of samples, which is defined as $A_{\textbf{x}}:\mathcal{H}_1\to \mathbb{R}^n$
\begin{equation*}
(A_{\textbf{x}} f)_j = <f,\phi_{X_j}>_{\mathcal{H}_1},
\end{equation*}
$\forall$ $j=1,\dots, n$ and $A_{\textbf{x}}^*$ is its adjoint operator given by 
\begin{equation*}
A^*_{\textbf{x}}\textbf{y} = \frac{1}{n}\sum_{j=1}^{n} Y_j \phi_{X_j}.
\end{equation*}
We want to highlight the difference with the usual statistical learning estimator and we denote the latter with 
\begin{equation}\label{stat learning estimator}
\hat{f}^{\lambda}_{n, \rm learn} = s_{\lambda}(A_{\textbf{x}}^* A_{\textbf{x}}) A_{\textbf{x}}^* \textbf{y}.
\end{equation}

Along the lines of inverse learning problems, we make the following assumptions.
\begin{assumption}\label{assumption statistical learning}
The $n$ observations $(X_i,Y_i)$ are i.i.d. drawn from a probability distribution $\rho$ on a Borel space $\mathcal{X}\times \mathcal{Y}$ so that $\nu$ is the marginal distribution of $X$ (see \cref{section learning}).
We assume 
that the conditional expectation with respect to $\rho(\cdot|\cdot)$ of $Y$ given $X$ is equal to
\begin{equation}
\label{rho_expectation}
\mathbb{E}(Y| X=x)=Af^{\dagger}(x)=y(x)
\end{equation}for $\nu$-almost $x\in\mathcal{X}$, where $f^{\dagger}\in\mathcal{H}_1$ is the generalized solution.
We assume also that the variance of the conditional probability is
\begin{equation}
\label{rho_variance}
Var(Y|X=x)=\sigma^2
\end{equation}  
for $\nu$-almost $x\in\mathcal{X}$, where $\sigma$ is a constant.
\end{assumption}
Now we are in the position to prove a first inequality between the expected error provided by $\hat{f}^{n}_{\lambda}$ as a function of $n$ and the error provided by $f^{\lambda}_{\delta}$. 
In what follows, to make it easier the writing, we do not write the subscript of the norms and we denote with $\mathbb{E}$ the mean computed with respect to the measure $\rho^{\otimes n}$ .
\begin{lemma}
\label{expected-reconstruction-error}
Let $\hat{f}_n^\lambda$ be defined in equation \cref{discretized-linear-regularization}. Under \cref{assumption statistical learning} we have
\begin{equation}\label{reconstruction error}
\mathbb E (\| \hat{f}_n^\lambda - f^\dagger \|^2) \ge \frac{\sigma^2}{n} 
\|L^\lambda\|_{HS}^2
+ \| f^\lambda - f^\dagger \|^2,
\end{equation}
where $\| \cdot \|_{HS}$ denotes the Hilbert Schmidt norm.
\end{lemma}
\begin{proof}
Denote with $\epsilon_n$ the difference between the estimate $\hat f^\lambda_n$ obtained with  $n$ samples and the sought solution $f^{\dagger}$. 
For any $t\in\mathcal{T}$ we have
\begin{eqnarray}
\epsilon^2_n(t) &=& \left( \frac{1}{n} \sum_{i=1}^n \ell^\lambda_{X_i}(t) Y_i - f^{\dagger}(t) \right)^2  \nonumber \\ 
&=& \frac{1}{n^2} \sum_{i,j=1}^{n} \ell^\lambda_{X_i}(t) Y_i \ell^\lambda_{X_j}(t) Y_j - \frac{2}{n} f^{\dagger}(t) \sum_{i=1}^n \ell_{X_i}(t) Y_i+ (f^{\dagger}(t))^2
\end{eqnarray}
By integrating over $\mathcal{Y}^n$, we get
\begin{eqnarray}
\int_{\mathcal{Y}^n} \epsilon^2_n(t) ~  d\rho(\cdot|\cdot)^{\otimes n}  &=&
\frac{1}{n^2} \sum_{i=1}^{n} (\ell_{X_i}^\lambda(t))^2 \sigma^2 + \frac{1}{n^2} \sum_{i,j=1}^{n} \ell^\lambda_{X_i}(t) \ell^\lambda_{X_j}(t) ~  y(X_i) y(X_j)  \nonumber \\
& & - \frac{2}{n} ~ f^{\dagger}(t) \sum_{i=1}^{n} \ell_{X_i}^\lambda(t) y(X_i) + (f^{\dagger}(t))^2,
\end{eqnarray}
where $d\rho(\cdot|\cdot)^{\otimes n}=d\rho(Y_1|X_1) \cdots d\rho(Y_n|X_n)$ and by using that $\rho(\cdot|\cdot)$ is a probability measure on $\mathcal{Y}$. 
Then, by integrating over $\mathcal{X}^n$ we obtain
\begin{eqnarray}
\int_{\mathcal{X}^n} \int_{\mathcal{Y}^n} \epsilon^2_n(t) ~ d\rho(\cdot|\cdot)^{\otimes n} d\nu^{\otimes n} 
&=&
\frac{\sigma^2}{n^2} \sum_{i=1}^{n} \int_{\mathcal{X}} (\ell^\lambda_{X_i}(t))^2 d\nu(X_i)  \nonumber \\
& & +
\frac{1}{n^2} \sum_{i=1}^{n^2-n} \left( \int_\mathcal{X} \ell^\lambda_{X_i}(t) y(X_i) d\nu(X_i) \right)^2\nonumber \\
& & +
\frac{1}{n^2} \sum_{i=1}^{n}  \int_\mathcal{X} \left(\ell^\lambda_{X_i}(t) y(X_i) \right)^2 d\nu(X_i) 
 \nonumber \\ & & - 
\frac{2}{n} f^{\dagger}(t) \sum_{i=1}^{n} \int_\mathcal{X} \ell^\lambda_{X_i}(t) y(X_i) d\nu(X_i) + (f^{\dagger}(t))^2  \nonumber \\
& \ge &
\frac{\sigma^2}{n^2} \sum_{i=1}^{n} \int_{\mathcal{X}} (\ell^\lambda_{X_i}(t))^2 d\nu(X_i)  \nonumber \\
& & +
\frac{1}{n^2} \sum_{i=1}^{n^2} \left( \int_\mathcal{X} \ell^\lambda_{X_i}(t) y(X_i) d\nu(X_i) \right)^2
 \nonumber \\ & & - 
\frac{2}{n} f^{\dagger}(t) \sum_{i=1}^{n} \int_\mathcal{X} \ell^\lambda_{X_i}(t) y(X_i) d\nu(X_i) + (f^{\dagger}(t))^2  \nonumber \\
&=&
\frac{\sigma^2}{n} \int_{\mathcal{X}} (\ell^\lambda_{X}(t))^2 d\nu(X)  \nonumber \\
& &
+ \left(f^\lambda(t)\right)^2  - 2 f^{\dagger}(t) f^\lambda(t) + (f^{\dagger}(t))^2 ~,
\end{eqnarray}
where we used that $\nu$ is a probability measure on $\mathcal{X}$.
Therefore, we have
\begin{eqnarray}
\mathbb{E} \left(\| \hat{f}^{\lambda}_n - f^{\dagger} \|^2\right) 
&\ge& 
\int_\mathcal{T}
\frac{\sigma^2}{n} \int_{\mathcal{X}} (\ell^{\lambda}_{X}(t))^2 ~ d\nu(X)
+
\left( f^\lambda(t) - f^{\dagger}(t) \right)^2  d\mu(t)
 \nonumber \\
&=& 
\frac{\sigma^2}{n} \| L^\lambda \|^2_{HS}
+ \| f^\lambda - f^{\dagger} \|^2 ~,
 \end{eqnarray}
as required.
\end{proof}

\begin{proposition}
\label{prop error}
Let us consider the inverse problem \cref{Af=y}, a family of linear regularization methods $L^\lambda$ and its discretized version \cref{discretized-linear-regularization}.
Then, for each $n \in \mathbb N$ there exists a function $\Delta(n,\lambda)$ such that for each $0<\delta\le \Delta(n,\lambda)$ and infinite dimensional noisy data $y^\delta$ such that $\| y^\delta - y \| \leq \delta$, the following inequality holds
 
\begin{equation}
\label{key}
\|f^\lambda_\delta - f^\dagger \|^2 \leq \mathbb E \left(\| \hat f^\lambda_n - f^\dagger \|^2\right) ~.
\end{equation}
Moreover,
\begin{equation}\label{Delta(n-lambda)}
\Delta(n,\lambda) = \frac{1}{\sqrt{\frac{\sigma^2}{n}+\varepsilon(\lambda)^2}+\varepsilon(\lambda)} \frac{\sigma^2}{n} ~,
\end{equation}
where $\varepsilon(\lambda) = \| f^\lambda - f^\dagger \|  \|L^\lambda\|_{HS}^{-1}$. 

Conversely, for each $\delta>0$ there exists a function $N(\delta,\lambda)$ such that for each $n\in\mathbb{N}$ such that $n\le N(\delta,\lambda)$ the equation \cref{key} applies and
\begin{equation}
\label{n(delta-lambda)}
N(\delta,\lambda)= \frac{\sigma^2}{\delta^2+2\delta\varepsilon(\lambda)}.
\end{equation}
\end{proposition}
\begin{proof}
We start from the result of \cref{expected-reconstruction-error}.
Easy manipulation of formula \cref{reconstruction error} leads to
\begin{equation}
\label{A}
\sqrt{\mathbb{E}\left( \| \hat{f}^{\lambda}_{n} - f^{\dagger} \|^2\right)} \ge \Delta(n, \lambda) \|L^\lambda\|_{HS} + \| f^\lambda - f^\dagger \|
\end{equation}
where $\Delta(n, \lambda)$ is defined as in equation \cref{Delta(n-lambda)}.
For each $\delta>0$, let $y^{\delta}$ s.t. $\| y^{\delta}-y\|\le\delta$, then a simple calculation gives
\begin{equation}\label{engl}
\| f^\lambda_\delta-f^{\dagger}\| \le \delta\| L^{\lambda}\| +\| f^{\lambda}-f^{\dagger}\|.
\end{equation}
Further, for each $\delta\le\Delta(n,\lambda)$ we have
\begin{equation}
\label{key2}
\sqrt{\mathbb{E} \left(\| \hat{f}^{\lambda}_{n} - f^{\dagger} \|^2\right)} \geq \delta \| L^{\lambda} \| +  \| f^{\lambda} - f^\dagger\|
\end{equation}
as $\| \cdot \|_{HS} \geq \| \cdot \|$.
From equations \cref{engl} and \cref{key2} we obtain $\forall$ $\delta\le \Delta(n,\lambda)$
\begin{equation}
\label{bound err_inverse_problems and risks}
\| f^\lambda_\delta - f^\dagger \|^2 \leq \mathbb{E}\left( \| \hat{f}^{\lambda}_{n} - f^{\dagger} \|^2\right)
\end{equation}
for each $y^{\delta}$ for which $\|y^{\delta} - y\| \leq \delta$.

Conversely, let $\delta>0$. 
For each $n\le N(\delta,\lambda)$, with $N(\lambda,\delta)$ defined by equation \cref{n(delta-lambda)} we have
\begin{equation}
\delta\le \Delta(n,\lambda)
\end{equation}
and so the thesis is proved.
\end{proof}

Functions $\Delta(n,\lambda)$ and $N(\delta,\lambda)$ express the dependency between the noisy level
$\delta$ and the number of samples $n$.
To make explicit this dependency we need to specify the rate of convergence of $\lambda \to 0$ both considered as a function of $\delta$ and $n$.
For the sake of convenience, we introduce the following
\begin{definition}
For any given $\lambda_n$ we define
\begin{equation}
\tilde{\delta}(n) := \Delta(n,\lambda_n) ~.
\end{equation}
Conversely, for any given $\lambda_\delta$  we define 
\begin{equation}
\tilde{n}(\delta) := \lfloor N(\delta,\lambda_{\delta}) \rfloor ~,
\end{equation}
where the symbol $\lfloor\cdot \rfloor$ denotes the integer part. 
\end{definition}
From now on, in order to express asymptotic behaviors we make use of the Landau symbols $O$, $\Omega$ and $\Theta$.



\begin{lemma}\label{lemma rates}
Let $\varepsilon(\lambda) \in \Theta(\lambda^{\gamma})$, with $\gamma\ge 0$.
If $\lambda_{n} \in \Theta(n^{-p})$, with $p>0$, then 
\begin{equation}\label{order delta_n}
\tilde{\delta}(n) \in \Theta\left(n^{-\max(\frac{1}{2},1-p\gamma)}\right) ~ .
\end{equation}
If $\lambda_{\delta} \in \Theta(\delta^{p^*})$, with $p^*>0$, then 
\begin{equation}\label{order n_delta}
\tilde{n}(\delta) \in \Theta\left(\delta^{-\min(2,p^*\gamma+1)}\right) ~ .
\end{equation}
\end{lemma}

\begin{proof}
The equation $\cref{order delta_n}$ follows from the definition of $\tilde{\delta}$ and from hypotheses $\lambda_n\in \Theta(n^{-p})$ and $\varepsilon(\lambda)\in \Theta(\lambda^{\gamma})$. In the same way the equation $\cref{order n_delta}$ follows from the definition of $\tilde{n}$ and from hypotheses $\lambda_{\delta}\in \Theta(\delta^{p^*})$ and $\varepsilon(\lambda)\in \Theta(\lambda^{\gamma})$.
\end{proof}



\begin{lemma}\label{lemma identity}
Given $\lambda_n$ there exists a unique $\lambda_{\delta}$ such that
\begin{equation}\label{identity delta n}
\tilde{\delta}\circ\tilde{n} =id_{\Im(\tilde{\delta})},
\end{equation}
where $id_{\Im(\tilde{\delta})}$ indicates the identity on the set $\Im(\tilde{\delta}) = \{\delta > 0 ~|~ \frac{\sigma^2}{\delta^2+2\delta\varepsilon(\lambda_{\delta})}\in\mathbb{N}\}$
and
\begin{equation}\label{Lambda_n respect Lambda_delta}
\Lambda^n = \Lambda^{\delta}\circ\tilde{\delta},
\end{equation}
where $\Lambda^{n} : \mathbb N \to \mathbb R$ and $\Lambda^{\delta} : \mathbb R \to \mathbb R$ are such that $\lambda_n = \Lambda^{n}(n)$ and $\lambda_{\delta} = \Lambda^{\delta}(\delta)$.
Furthermore, 
\begin{equation}\label{identity n delta}
\tilde{n}\circ\tilde{\delta}=id_{\mathbb{N}}. 
\end{equation}
\end{lemma}
\begin{proof}
The existence and uniqueness of $\lambda_{\delta}$ such that \cref{identity delta n,Lambda_n respect Lambda_delta} are verified follow by defining $\lambda_{\delta}:= \Lambda^n(\tilde{n}(\delta))$. 
With straightforward calculus it can be verified that \cref{Lambda_n respect Lambda_delta} implies \cref{identity n delta}. 
\end{proof}
Similarly, we give the converse result. 
\begin{lemma}\label{conversely lemma identity} 
Given $\lambda_{\delta}$, there exists a unique $\lambda_n$ such that 
\begin{equation*}
\tilde{n}\circ\tilde{\delta} = id_{\mathbb{N}}
\end{equation*}
and
\begin{equation}
\Lambda^{\delta} = \Lambda^n\circ\tilde{n}
\end{equation}
where we have used the same notation of \cref{lemma identity}. Furthermore, \begin{equation}
\tilde{\delta}\circ\tilde{n} = id_{\Im(\tilde{\delta})}.
\end{equation}
\end{lemma}
The proof is analogous to the one of \cref{lemma identity} by defining $\lambda_n = \Lambda^{\delta}(\tilde{\delta}(n))$.
The following result relates a given upper convergence rate computed with respect to $n$ to the one computed with respect to $\delta$. 
\begin{theoremmm}\label{theo_conv_rates}
Let the average upper rate with respect to the number of samples identically and independently drawn according to a distribution $\rho$ be equal to $n^{-\alpha}$ for a given $\alpha> 0$, i.e.
\begin{equation}
\label{stochastic-convergence-rate}
\mathbb E (\| \hat{f}_n^\lambda - f^\dagger \|^2)\in O  \left(\frac{1}{n}\right)^{\alpha} ~ ,
\end{equation}
given $\lambda =\lambda_n=\Theta\left(n^{-p}\right)$, with $p>0$ and $\varepsilon(\lambda)=\Theta\left(\lambda^{\gamma}\right)$, with $\gamma>0$. 
Then the upper rate of the error with respect to the noise level $\delta \to 0$ is given by
 \begin{equation}\label{order deterministic error}
\| f^\lambda_\delta - f^\dagger \|^2 \in
\begin{cases}
O\left(\delta^{2\alpha}\right) & \text{ if } p\gamma\ge\frac{1}{2} \\
O\left( \delta^{\frac{\alpha}{1- p\gamma}} \right) & \text{ if } p\gamma <\frac{1}{2}~ ,
\end{cases}
\end{equation}
where $f^{\lambda}_{\delta}$ is defined in equation \cref{solution f lambda delta}, 
$y^\delta$ is such that $\|y^\delta - y\|\le\delta $, and $\lambda=\lambda_{\delta}$ is defined in \cref{lemma identity} and it has the following rate
\begin{equation}\label{order lambda_delta}
\lambda_{\delta}\in 
\begin{cases}
\Theta\left(\delta^{2p}\right) & \text{ if } p\gamma\ge\frac{1}{2} \\
\Theta\left(\delta^{\frac{p}{1-p\gamma}}\right) & \text{ if } p\gamma <\frac{1}{2}.
\end{cases}
\end{equation}
\end{theoremmm}
\begin{proof}
Given $\lambda_n=\Lambda^{n}(n)$, we define $\lambda_{\delta}=\Lambda^{\delta}(\delta)$ according to \cref{lemma identity}, so that equations 
\cref{identity delta n} and \cref{Lambda_n respect Lambda_delta} hold.
The rate of $\lambda_{\delta}$ given in \cref{order lambda_delta} can be found by using the hypothesis $\lambda_n=\Theta(n^{-p})$ and \cref{lemma rates}.
Now we prove \cref{order deterministic error}. Thanks to \cref{prop error} and \cref{lemma rates}, for each $\lambda>0$ and $\delta>0$ there exists $\tilde{n}(\delta)$ such that for all $n\le\tilde{n}(\delta)$
\begin{equation}
\| f^\lambda_\delta - f^{\dagger}\|^2\le\mathbb{E}(\| \hat{f}_n^{\lambda}-f^{\dagger}\|^2).
\end{equation}
Let $n=\tilde{n}(\delta)$, then
\begin{equation}
\| f^{\lambda}_{\delta} - f^{\dagger}\|^2\le\mathbb{E}(\| \hat{f}_{\tilde{n}(\delta)}^{\lambda}-f^{\dagger}\|^2).
\end{equation}
Let $\lambda=\lambda_{\delta}$.
Then there exist $n_0\in\mathbb{N}$ and $M>0$ such that
\begin{equation}
\label{order error}
\begin{split}
\| f^{\lambda_{\delta}}_{\delta}-f^{\dagger}\|^2 & \le \mathbb{E}(\| \hat{f}_{\tilde{n}(\delta)}^{\Lambda^{\delta}(\delta)}-f^{\dagger}\|^2) 
 = \mathbb{E}(\| \hat{f}_{\tilde{n}(\delta)}^{\Lambda^{n}(\tilde{n}(\delta))}-f^{\dagger}\|^2)
\le M \left(\frac{1}{\tilde{n}(\delta)}\right)^{\alpha},
\end{split}
\end{equation}
for all $\tilde{n}(\delta)>n_0$.
From equations \cref{order lambda_delta} and \cref{order error}, and by using \cref{lemma rates} we obtain
\begin{itemize}
\item[-] if $p\gamma\ge\frac{1}{2}$ then $\lambda_{\delta}=\Theta(\delta^{2p})$, therefore from \cref{prop error} we have $\tilde{n}(\delta)\in\Theta(\delta^{2})$ and from equation \cref{order error} we obtain $\| f^\lambda_\delta -f^{\dagger}\|^2\in O(\delta^{2\alpha})$
\item[-] if $p\gamma<\frac{1}{2}$ then $\lambda_{\delta}\in \Theta(\delta^{\frac{p}{1-p\gamma}})$, therefore from \cref{prop error} we have $\tilde{n}(\delta)\in\Theta(\delta^{\frac{1}{1-p\gamma}})$ and from equation \cref{order error} we obtain $\| f^\lambda_\delta-f^{\dagger}\|^2\in O(\delta^{\frac{\alpha}{1-p\gamma}})$.
\end{itemize}
This completes the proof.
\end{proof}

Now we give the converse result regarding lower rates. 
\begin{theoremmm}\label{conversely theorem rates}
Let $f^{\lambda}_{\delta}$ be defined in \cref{solution f lambda delta}. Let the lower rate of the convergence error with respect to the noise level $\delta \to 0$ be equal to $\delta^{\alpha}$ for a given $\alpha>0$, i.e.
 \begin{equation}\label{error rate conversely}
\| f^{\lambda}_{\delta} - f^\dagger \|^2 \in\Omega(\delta^{\alpha}),
\end{equation}
where  $y^{\delta}$ is such that $\| y^{\delta}-y\|\le\delta$, $\lambda=\lambda_{\delta}=\Theta(\delta^{p^*})$, with $p^*>0$ and $\varepsilon(\lambda)=\Theta(\lambda^{\gamma})$, with $\gamma>0$. Then the average lower rate with respect to the number of samples $n\to\infty$ is given by
\begin{equation}
\label{stochastic lower rate}
\mathbb E (\| \hat{f}_n^\lambda - f^\dagger \|^2)\in
\begin{cases}
\Omega\left(n^{-\frac{\alpha}{2}}\right) & \text{ if } p^*\gamma\ge 1 \\
\Omega\left(n^{-\frac{\alpha}{1+p^*\gamma}}\right) & \text{ if } p^*\gamma< 1
\end{cases}
\end{equation} 
where $\lambda=\lambda_n$ is defined in \cref{conversely lemma identity} and it has the following rate
\begin{equation}\label{order lambda_n}
\lambda_{n}\in 
\begin{cases}
\Theta\left(n^{-\frac{p^*}{2}}\right) & \text{ if } p^*\gamma\ge 1 \\
\Theta\left(n^{-\frac{p^*}{1+p^*\gamma}}\right) & \text{ if } p^*\gamma < 1.
\end{cases}
\end{equation}
\end{theoremmm}
\begin{proof}
The proof exploits a similar argument to the one used for \cref{theo_conv_rates}. Given $\lambda_{\delta}=\Lambda^{\delta}(\delta)$, by defining $\lambda_n =\Lambda^{n}(n)$ according to \cref{conversely lemma identity}, it can be proved that the rate of $\lambda_n$ is given by equation \cref{order lambda_n}. 
To prove \cref{stochastic lower rate} one has to reverse the role of $n$ and $\delta$ in the proof of \cref{theo_conv_rates} and use \cref{prop error} and hypothesis \cref{error rate conversely}. In such a way one obtains that for each $n\in\mathbb{N}$, there exist $\delta_0>0$ and $M'>0$ such that
\begin{equation}
\label{lower error}
\mathbb{E}(\| \hat{f}_n^{\lambda_n}-f^{\dagger}\|^2)\ge
\| f^{\Lambda^{\delta}(\tilde{\delta}(n))}_{\tilde{\delta}(n)}-f^{\dagger}\|^2\ge M' (\tilde{\delta}(n))^{\alpha} ~ ,
\end{equation}
for all $\tilde{\delta}(n)<\delta_0$.
The thesis follows from \cref{order lambda_n,lower error,lemma rates}.
\end{proof}

\subsection{Convergence rates in the statistical setting}\label{subsection rate of our estimator}
In the statistical learning framework a lot of research is devoted to investigate convergence rates \cite{blanchard2017optimal,bauer2007regularization,lin2018optimal}. In this section we show that, under the same hypotheses considered in \cite{blanchard2017optimal}, the upper rate of the estimator $\hat{f}^{\lambda}_n$ defined in \cref{hat f} is of the same order of the one of the classical spectral estimator \cite{blanchard2017optimal}. 
First we describe the usual assumptions considered in the study of convergence rates. 
Source conditions are expressed in terms of restrictions of the probability space, and they correspond to assume a certain degree of smoothness of the infinite dimensional solutions and operators.
The first restriction applies to the smoothness of the sought solutions.
We assume that the solution belongs to the set
\begin{equation}
\label{eq:source-conditions}
\omega(r,R):=\{f\in\mathcal{H}_1 ~ : ~ f = B^{r} w, ~ \|w\|_{\mathcal{H}_1}\le R \}, 
\end{equation}
where $B:=A^*A,$ with $A^*$ the adjoint operator and $r>0$ and $R>0$.
In the statistical framework this assumption is given as a requirement on the probability $\rho(\cdot|\cdot)$.
In particular, $\rho(\cdot|\cdot)$ has to be such that \cref{rho_expectation} holds and the sought solution belongs to the set $\omega(r,R)$. 
The second restriction applies to the eigenvalue decay of the operator $B$. 
We assume that
\begin{equation}
\label{eq:assumption1}
\mu_j \leq \frac{d}{j^b}
\end{equation}
where $\mu_j$ are the eigenvalues of $B$ for each $j \in \mathbb N$, $j\ge 1$, $d > 0$ and $b>1$. 
In the statistical framework this assumption is given as a requirement on the probability $\nu$ which $B$ depends on. 
These two assumptions are restrictions on $\rho(\cdot|\cdot)$ and $\nu$ respectively and they are summarized as a single restriction on the probability space by requiring that $\rho$ given by equation \cref{rho} belongs to a suitable subspace $\mathcal M(r,R,b)$ of the probability space where $\mathcal M(r,R,b)$ represents the class of models (for details see \cite{blanchard2017optimal}). 

We give an upper rate for $\mathbb{E}(\Vert \hat{f}^{\lambda}_n-f^{\dagger}\Vert^2)$ by exploiting the upper bound given in \cite{bissantz2007convergence},  where a more general mixed type noise model is considered and the stochastic part of the noise is modeled as an Hilbert-space process.
We recall the properties of the regularization function $s_{\lambda}$. 
\begin{definition}
The regularization (or filtering) functions $s_{\lambda}$ for $\lambda>0$ defined on the spectrum of $A^*A$, denoted by $\sigma(A^*A)$, have to satisfy the following properties 
\begin{itemize}
\item there exists a constant $D>0$ such that
\begin{equation}\label{property1 s_lambda}
\sup_{t\in\sigma(A^*A)} |t s_{\lambda}(t)|\le D ~ ~ \text{ uniformly in } ~ \lambda>0
\end{equation}
\item there exists a constant $E>0$ such that\
\begin{equation}\label{property2 s_lambda}
\sup_{\lambda>0} \sup_{t\in\sigma(A^*A)}|\lambda s_{\lambda}(t)|\le E
\end{equation}
\item there exists $q>0$ called qualification of the method and constants $C_{\nu}>0$ such that
\begin{equation}\label{qualification reg function}
\sup_{t\in\sigma(A^*A)} |t^{\nu}(1-t s_{\lambda}(t))| \le C_{\nu} \lambda^{\nu} ~ ~ \forall ~ \lambda>0 ~ \text{ and } ~ ~ 0\le\nu\le q.
\end{equation}
\end{itemize}
\end{definition}

We remark that $\mathbb{E}(\Vert \hat{f}^{\lambda}_n-f^{\dagger}\Vert^2)$ satisfies the bias-variance decomposition as follows
\begin{equation}\label{bias-variance decomposition}
\mathbb{E}(\| \hat{f}^{\lambda}_n-f^{\dagger}\|^2)= B(\hat{f}_{\lambda})^2 + \mathbb{E}(\| \hat{f}_{\lambda}-\mathbb{E}(\hat{f}_{\lambda})\|^2),
\end{equation}
where $B(\hat{f}_{\lambda}):=\| \mathbb{E}(\hat{f}^{\lambda}_n)-f^{\dagger}\|$ is the bias term and $\mathbb{E}(\hat{f}^{\lambda}_n)=f^{\lambda}$.
Under the source condition \cref{eq:source-conditions} the bias term can be bounded by
\begin{equation}
B(\hat{f}_{\lambda}) \le C_{r} \lambda^r R,  
\end{equation}
where $C_{r}$ is the constant of the property \cref{qualification reg function} of the regularization function $s_{\lambda}$. Hereafter, we consider $r\le q$.
The estimation of the variance term needs more manipulations. In the following result we show the optimal upper rate achieved by $\hat{f}^{\lambda}_n$.
\begin{lemma}
Let $\hat{f}^{\lambda}_{n}$ be defined in \cref{hat f} and let the model be described by \cref{rho_expectation,rho_variance}. Under the source conditions \cref{eq:source-conditions,eq:assumption1} we have
\begin{equation}\label{upper rate estimator}
\mathbb{E}(\| \hat{f}_n^{\lambda}-f^{\dagger}\|^2)\in O\left(\left(\frac{1}{n}\right)^{\frac{2r}{2r+1+\frac{1}{b}}}\right),
\end{equation}
with $\lambda \in \Theta\left((\frac{1}{n})^{\frac{1}{2r+1+\frac{1}{b}}}\right)$. 
\end{lemma}
\begin{proof}
We follow the argument given in the section 4.3 in \cite{bissantz2007convergence}. We define $\tilde{\epsilon}$ as an Hilbert-space noise process such that $A^* \tilde{\epsilon}=\textbf{A}_{\textbf{x}}\textbf{y}-A^*Af^{\dagger}$.
The noise $\epsilon=\tilde{\sigma}\tilde{\epsilon}$, where $\tilde{\sigma} = \frac{\sqrt{C}}{\sqrt{n}}$ with $C$ a constant depending on the variance $\sigma^2$, satisfies the assumption of the Theorem 3 in \cite{bissantz2007convergence}. Then, we have
the following bound
\begin{equation}\label{bissantz ineq}
\mathbb{E}(\| \hat{f}^{\lambda}_n-\mathbb{E}(\hat{f}^{\lambda}_n) \|^2)=\mathbb{E}(\| s_{\lambda}(A^*A) A^* \tilde{\sigma}\epsilon \|^2) \le \frac{C}{n} L \frac{1}{\lambda^{2}} \int_{0}^{\lambda} \beta^{-\frac{1}{b}} d\beta = \frac{C}{n} L \frac{1}{\lambda^{1+\frac{1}{b}}},
\end{equation}
under assumption \cref{eq:assumption1} and where $L$ is a constant which depends on $D$ and $E$ (see properties \cref{property1 s_lambda,property2 s_lambda}) and constants in the assumption \cref{eq:assumption1}. 
Therefore, under assumption \cref{eq:source-conditions} we obtain
\begin{equation}\label{upper bound estimator}
\mathbb{E}(\| \hat{f}_n^{\lambda}-f^{\dagger}\|^2)\le C^2_r \lambda^{2r} R^2 + \frac{C}{n} L \frac{1}{\lambda^{1+\frac{1}{b}}}.
\end{equation}
By balancing terms in the r.h.s. of \cref{upper bound estimator} we have the thesis.
\end{proof}
Then, the upper rate given in \cref{upper rate estimator} is the same of the classical spectral estimator $\hat{f}^{\lambda}_{n,\rm learn}$ defined in \cref{stat learning estimator}. 

\subsection{Conversion of convergence rates}\label{Application of conversion rates in the spectral theory}

Under assumptions \cref{eq:source-conditions,eq:assumption1} we use \cref{theo_conv_rates} to transform the upper rate \cref{upper rate estimator} to an upper rate for the classical spectral regularization depending on $\delta$. Let  $f^{\lambda}_{\delta}:=s_{\lambda}(A^*A)A^*y^{\delta}$ and $\gamma$ be defined as in \cref{theo_conv_rates} then we have two cases:
\begin{itemize}
\item if $\frac{\gamma}{2r+1+\frac{1}{b}}\ge\frac{1}{2}$
\begin{equation}\label{best case conversion n to delta}
\Vert f^{\lambda}_{\delta}-f^{\dagger}\Vert \in O\left(\delta^{\frac{2r}{2r+1+\frac{1}{b}}}\right)
\end{equation}
where $\lambda \in \Theta\left(\delta^{\frac{2}{2r+1+\frac{1}{b}}}\right)$,
\item if $\frac{\gamma}{2r+1+\frac{1}{b}}<\frac{1}{2}$
\begin{equation}
\Vert f^{\lambda}_{\delta}-f^{\dagger}\Vert \in O\left(\delta^{\frac{r}{2r+1+\frac{1}{b}-\gamma}}\right)
\end{equation}
where $\lambda\in \Theta\left(\delta^{\frac{1}{2r+1+\frac{1}{b}-\gamma}}\right)$.
\end{itemize}
The first case gives a faster rate with respect to the second one.
We remark that in both cases the obtained upper rate is slower than the classical optimal one, i.e. $O\left(\delta^{\frac{2r}{2r+1}}\right)$ \cite{engl_regularization_1996}. 
The ratio between the first best case in \cref{best case conversion n to delta} and the classical optimal one $O\left(\delta^{\frac{2r}{2r+1}}\right)$ is equal to $\tau=\frac{2r+1+\frac{1}{b}}{2r+1}>1$. $\tau$ represents the loss factor converting the optimal rate in the statistical setting to a rate depending on the noise level $\delta$. We have 
$\tau < 2$
and $\tau$ is close to 1 when $b$ is large, which means that the eigenvalues decay \cref{eq:assumption1} has to be fast. 

\begin{example}
We consider the Tikhonov solution. In this case we have $\gamma=r+\frac{1}{2}$. 
\begin{itemize}
\item Upper rate conversion. Using the upper rate \cref{upper rate estimator} and \cref{theo_conv_rates} we have
\begin{equation}
\Vert f^{\lambda}_{\delta,\rm Tik}-f^{\dagger}\Vert \in O\left(\delta^{\frac{r}{r+\frac{1}{2}+\frac{1}{b}}}\right),
\end{equation}
where $\lambda\in\Theta\left(\delta^{\frac{1}{r+\frac{1}{2}+\frac{1}{b}}}\right)$. By comparing with the classical optimal rate $O\left(\delta^{\frac{2r}{2r+1}}\right)$, 
we have that the loss factor is $\tau_{\rm Tik}=\frac{2r+1+\frac{2}{b}}{2r+1}<3$. 
\item Lower rate conversion. The classical lower rate established for Tikhonov solution with respect to the noise level $\delta$, i.e. $\Vert f^{\lambda}_{\delta,\rm Tik}-f^{\dagger}\Vert \in \Omega\left(\delta^{\frac{2r}{2r+1}}\right)$, under the source condition \cref{eq:source-conditions} can be converted to a lower rate with respect to the number of samples $n$ using \cref{conversely theorem rates}, that is
\begin{equation}\label{conv lower rate tik}
\mathbb{E}(\Vert \hat{f}_{n,\rm Tik}^{\lambda}-f^{\dagger}\Vert^2)\in\Omega\left(\left(\frac{1}{n}\right)^{\frac{2r}{2r+1}}\right),
\end{equation}
where $\lambda\in\Theta\left(\left(\frac{1}{n}\right)^{\frac{1}{2r+1}}\right)$.
In \cite{blanchard2017optimal} it is proven that $\left(\frac{1}{n}\right)^{\frac{2r}{2r+1+\frac{1}{b}}}$ is a minimax rate for the estimator $\hat{f}^{\lambda}_{n,\rm learn}$. 
The lower bound is given under the source condition \cref{eq:source-conditions} and the following hypothesis on the eigenvalues decay of the operator $B$
\begin{equation}\label{eigen lower}
\mu_j\ge\frac{a}{j^b}
\end{equation}
where $a>0$, $b>1$ and $j\ge 1$.
Comparing the result in \cref{conv lower rate tik} and the lower rate established in \cite{blanchard2017optimal} we note that the lower rate in \cref{conv lower rate tik} is retrieved as $b\to\infty$. In such a case the r.h.s of the eigenvalue condition \cref{eigen lower}
vanishes for all $j \geq 2$ when $b \to \infty$. Thus, it takes the form $\mu_1 \geq a>0$ which can be read as a condition $B \neq 0$. 
\end{itemize} 
\end{example}

\ifdefined\OK
We now show that Theorem 3.5 in \cite{blanchard2017optimal} is related to the classical results on the optimality order in\cite{engl_regularization_1996}.
Let us consider, in place of condition \cref{eq:assumption1}, the condition 
\begin{equation}
\frac{\alpha}{j^b} \leq \mu_j
\end{equation}
for some small $\alpha$.
Under this assumption, Theorem 3.5 in \cite{blanchard2017optimal} proved the weak minmax lower rate of convergence, i.e.
\begin{equation}
\inf_{R>0} ~ \limsup_{n \to \infty} ~ \inf_{\lambda} \sup_{\rho \in \mathcal M(r,b,R)} \frac{\sqrt{ \mathbb  E_{\rho^{\otimes n}} \| f^\lambda_n - f_\rho \|^2} }{ a_n} > 0 ~ .
\end{equation}
By construction, taking the supremum over the probabilities $\rho \in \mathcal M(r,b,R)$ corresponds to taking the supremum over the solutions $f_\rho \in \Omega(\mu,r,R)$.
The infimum over $\lambda$ corresponds to select the method of the family $L^\lambda$ which provides the best error in the worst-expected case.
Taking the superior limit ensures that the error cannot be of order lower than $a_n$. 
The infimum over the radii $R$ corresponds to consider the set
\begin{equation}
\Omega(r) = \bigcup_{R>0} \Omega(r,R) 
\end{equation}

\fi

\section{Conclusion}

In this paper we attempted to give a uniform vision of discrete inverse problems and supervised learning.
We started from the infinite dimensional approximation problem in a RKHS, showing that there is a natural correspondence between its solution and the solutions of a certain class of inverse problems.
Such a correspondence suggests that these problems are equivalent to some extent: we showed that as well as the {\it data space} of a linear inverse problem is a RKHS, 
the feature space of a learning problem can be thought of as the {\it parameter space} of a linear inverse problem.
Then, we distinguished learning and discrete inverse problems according to a different discretization scheme of the same infinite dimensional problem.
We analyzed the convergence of the discretized functionals and solutions to their corresponding ideal ones, and in the case of a deterministic discretization we gave some mild sufficient conditions to have the convergence relying on the $\Gamma$-convergence theory.
Finally, we investigated the connection between error convergence rates in the case the error is computed as a function of the noise level $\delta$ and as a function of the number of examples $n$.
We quantified the deviation between optimal rates in the two frameworks.

\bibliographystyle{siamplain}
\bibliography{sample.bib}

\end{document}